\documentclass[twoside]{scrartcl}
\usepackage[nouppercase,headsepline]{scrpage2}

\usepackage{tikz-cd}
\usetikzlibrary{decorations.pathmorphing}
\usepackage{mathtools}  
\usepackage{amsmath}
\usepackage{amssymb}
\usepackage{amsthm}
\usepackage{amsxtra}

\usepackage[english]{hyperref}
\usepackage[english]{babel}
\usepackage[utf8]{inputenc}

\usepackage{url}
\usepackage[numbers]{natbib}

\newcommand*{\Q}{\mathbb{Q}}
\newcommand*{\R}{\mathbb{R}}
\newcommand*{\C}{\mathbb{C}}
\newcommand*{\Z}{\mathbb{Z}}

\newcommand*{\fieldk}{\mathbf{k}}

\newcommand*{\Dltc}{D}

\newcommand*{\Disc}{D_{\fieldk}} 
\newcommand*{\order}[1]{\mathcal{O}{\left(#1\right)}}
\newcommand*{\OK}{\mathcal{O}_{\fieldk}}
\newcommand*{\Diffk}{\mathfrak{d}_\fieldk}
\newcommand*{\sqrtD}{\delta_{\fieldk}}

\newcommand*{\group}[1]{\mathrm{#1}}

\newcommand*{\SU}{\group{SU}}
\newcommand*{\Ug}{\group{U}}
\newcommand*{\Og}{\group{O}}

\newcommand*{\SL}{\group{SL}}

\newcommand*{\Mp}{\group{Mp}}

\newcommand*{\HeisU}[1]{\mathrm{Heis}_{#1}}

\newcommand*{\Vk}{V_\fieldk}

\newcommand*{\HU}{\mathcal{H}}
\newcommand*{\DomU}{\mathbb{D}} 
\newcommand*{\Hp}{\mathbb{H}}

\newcommand*{\coneU}{\mathcal{C}}

\newcommand*{\projV}{\pi_V} 

\DeclarePairedDelimiter{\hlfa}{\langle}{\rangle}
\newcommand*{\hlf}[2]{\hlfa*{ #1, #2}} 
\newcommand{\hlfempty}{\hlf{\cdot}{\cdot}}
\DeclarePairedDelimiter{\blfp}{(}{)}
\newcommand*{\blf}[2]{\blfp*{ #1, #2}} 
\newcommand*{\blfempty}{\blf{\cdot}{\cdot}} 

\newcommand*{\Qf}[1]{q\!\left( #1 \right)} 
\newcommand*{\QfNop}{q} 
\DeclarePairedDelimiter{\abs}{\lvert}{\rvert}
\newcommand*{\Iso}{\mathrm{Iso}}

\DeclareMathOperator{\tr}{Tr}

\newcommand*{\Ueps}{U_{\epsilon}}
\newcommand*{\Oeps}{\mathcal{O}_\epsilon}
\newcommand*{\Peps}{\mathcal{P}_\epsilon}
\newcommand*{\Pic}{\mathrm{Pic}}

\newcommand*{\KH}{\mathrm{H}}
\newcommand*{\Htwo}{\KH^2} 
\newcommand*{\CoC}{\mathrm{C}}
\newcommand*{\Hom}{\operatorname{Hom}}

\newcommand*{\BIL}{\mathsf{BIL}}
\newcommand*{\BilZ}{\BIL_\Z}

\newcommand*{\trg}{\operatorname{tg}}
\newcommand*{\Lcal}{\mathcal{L}}
\newcommand*{\Dg}{D_{\ell, \Gamma}}
\newcommand*{\Ng}{N_{\ell,\Gamma}}

\newcommand*{\HeegU}{\mathbf{H}}

\newcommand*{\ebase}{\mathfrak{e}} 
\newcommand*{\Cuspf}{\mathcal{S}}

\newtheorem{theorem}{Theorem}[section]
\newtheorem{definition}{Definition}[section]
\newtheorem{proposition}{Proposition}[section]
\newtheorem{lemma}{Lemma}[section]

\newtheorem{rmk}{Remark}[section]
\newtheorem*{rmk*}{Remark}

\newtheorem{corollary}{Corollary}[section]
\newtheorem*{assumption*}{Assumption}

\newcommand*{\acknowledegements}{\textit{Acknowledegements:}\ }

\clearscrheadfoot
\lohead{}\cohead{}\rohead{Local Borcherds Products for Unitary Groups} 
\lehead{}\cehead{}\rehead{Eric Hofmann}
\title{Local Borcherds Products for Unitary Groups}
\author{Eric Hofmann
}\date{}
\KOMAoptions{abstract=on}
\begin{document}
\maketitle
\begin{abstract} 
For the modular variety attached to an arithmetic subgroup  of an indefinite unitary group of signature $(1,n+1)$, with $n\geq 1$, 
we study Heegner divisors in the local Picard group over a boundary component of a compactification. 
For this purpose, we introduce local Borcherds products. 
We obtain a precise criterion for local Heegner divisors to be torsion elements in the Picard group, 
and further, as an application, we show that the obstructions to a local Heegner divisor being a torsion element
can be described by certain spaces of vector valued elliptic cusp forms, transforming under a Weil-representation. 
\footnote{\textit{Date:} \today \\%
 }
\end{abstract}
\setcounter{section}{-1}
\section{Introduction and statement of results}
A  \emph{local Borcherds}  product is a holomorphic function, which, like a Borcherds form has an absolutely convergent infinite product expansion
and an arithmetically defined divisor, called a local Heegner divisor.  Here, \lq local\rq\ refers to boundary components of a modular variety. 
Such products were first introduced by Bruinier and Freitag, who, in \cite{BrFr} studied the local divisor class groups of generic boundary components for the modular varieties  of indefinite orthogonal groups $\Og(2,l)$, $l\geq 3$.
Since then, local Borcherds products have appeared in several places in the literature,
for example in \cite{B123}, for the Hilbert modular group, and  in \cite{FS}, 
where they are introduced to study a specific problem in the geometry of Siegel three folds.

The aim of the present paper is to develop a theory similar to that of Bruinier and Freitag
for unitary groups of signature $(1, n+1)$, $n\geq 1$. 

Let $\fieldk = \Q(\sqrt{\Disc})$ be an imaginary quadratic number field with discriminant $\Disc$, 
which we consider as a subset of $\C$.
Denote by $\OK$ the ring of integers in $\fieldk$, by $\Diffk^{-1}$ the inverse different ideal 
and by $\sqrtD$ the square-root of $\Disc$, with the principal branch of the complex square-root.

Let $V$ be an indefinite hermitian vector space over $\fieldk$ of signature $(1, n+1)$,
equipped with a non-degenerate hermitian form $\hlfempty$. 
Let $L$ be a lattice in $V$, of full rank as an $\OK$-module,  
so that $L\otimes_{\OK} \fieldk = V$. We assume that $L$ is an even and integral lattice,
hence $\hlf{\lambda}{\lambda} \in \Z$ for all $\lambda \in L$. 
In this introductory section only, we additionally assume that $L$ is unimodular over $\Z$, 
i.e.\  $L = L' = \left\{ \mu \in V;\; \hlf{\lambda}{\mu}\in\Diffk^{-1},\, \forall \lambda \in L \right\}$. 

We denote by $\Ug(V)$ the unitary group of $V$ and by $\Ug(L) \subset \Ug(V)$ 
the isometry group of $L$.  Subgroups of finite index in $\Ug(L)$ are called unitary modular groups.

We consider $\Ug(V)$ as an algebraic group defined over $\Q$. Its set of real points, denoted $\Ug(V)(\R)$, is the unitary group of the complex hermitian space 
$V \otimes_\fieldk \C$. A symmetric domain for the operation of this group is given by the quotient 
\[ 
\DomU = \Ug(V)(\R) / \mathcal{K},
\]
where $\mathcal{K}$ is a maximal compact subgroup of $\Ug(V)(\R)$.
If $\Gamma \subset \Ug(L)$ is a unitary modular group, we denote by $X_\Gamma$ the modular variety given by the quotient $\Gamma \backslash \DomU$. 
Note that $X_\Gamma$ is non-compact. 

The boundary points of $\DomU$ correspond one-to-one 
to the elements $I$ of the set of rational one-dimensional  isotropic subspaces of $V$, denoted $\Iso(V)$. 
For every cusp of  $X_{\Gamma}$ one can thus introduce a small open neighborhood $U_\epsilon(I)$. 
These neighborhoods are then glued to $X_\Gamma$, furnishing a compactification. 
We describe this procedure in section \ref{subsec:cmpct} both for the Baily-Borel compactification, in which singularities remain at the cusps, 
and for a toroidal compactification, which turns $X_\Gamma$ into a normal complex space without singularities at the cusps. 

We will study the  Picard groups of such (suitably small) open neighborhoods $U_\epsilon(I)$. 
Since the construction we carry out is local in nature, it suffices to examine only one fixed cusp.
For this purpose, we choose a primitive isotropic lattice vector $\ell \in L$. Fixing a vector $\ell' \in L$ with $\hlf{\ell}{\ell'} \neq 0$,  denote 
by $\Dltc$ the definite lattice $L\cap \ell^\perp \cap \ell'^\perp$. 
The stabilizer $\operatorname{Stab}_\Gamma(\ell)$ of $\ell$ in $\Gamma$ contains a Heisenberg group, denoted $\Gamma_\ell$. 
This group has finite index in the stabilizer. Its elements can be written as pairs $[h,t]$, with $h$ a rational number and $t$ a lattice vector. The set of all such $t$'s constitutes a sub-lattice $\Dg \subseteq \Dltc$.

Following \cite{BrFr}, we define the Picard group $\Pic(X_\Gamma, \ell)$ as the direct limit $\varinjlim\Pic(\Ueps^{reg}(\ell))$, 
where $\Ueps^{reg}(\ell)$ is the regular locus of $\Ueps(\ell)$ in the Baily-Borel compactification. 

Up to torsion, this local Picard group can also be described by the direct limit 
$\varinjlim \Pic(\Gamma_\ell \backslash \Ueps(\ell))$, see p.\ \pageref{sec:heegner} for details.
Thus, if we only want to describe the position of certain special divisors in  $\Pic(X_\Gamma, \ell)$ up to torsion,
we can work in  $\Pic(\Gamma_\ell \backslash U_\epsilon(\ell))$, with a sufficiently small $\epsilon >0$.

For a lattice vector  $\lambda \in L$ of negative norm, i.e.\  $\hlf{\lambda}{\lambda} \in \Z_{< 0}$, 
a primitive Heegner divisor $\HeegU(\lambda)$ is defined by the orthogonal  complement $\lambda^\perp$ with
respect to $\hlfempty$ of $\lambda$ in $\DomU$. 
If $\ell$ lies in $\lambda^\perp$, we attach a local Heegner divisor to $\lambda$ by setting 
$\HeegU_\infty(\lambda) \vcentcolon= \sum_{ \alpha \in \Diffk^{-1}} \HeegU(\lambda+ \alpha\ell)$.

A Heegner divisor of $\DomU$ is a $\Gamma$-invariant finite linear combination 
of primitive Heegner divisors and the pre-image under the canonical projection of a divisor on $X_\Gamma$.  
By a \emph{local Heegner divisor}, we mean a finite linear combination of local Heegner divisors of the form $\HeegU_\infty(\lambda)$,
which corresponds to the pre-image of an element of the divisor group 
$\mathrm{Div}( \Gamma_\ell \backslash U_\epsilon(\ell))$, see section \ref{subsec:heegner} for details.

We want to describe the position of local Heegner divisors in the local Picard group $\Pic(X_\Gamma, \ell)$ (up to torsion)
through their position in $\Pic(\Gamma_\ell\backslash U_\epsilon(\ell))$.
This is where local Borcherds products come into play: 
For a negative-norm lattice vector $\lambda$ we define the local Borcherds product $\Psi_\lambda(z)$
as follows (see section \ref{subsec:localbp}):
\[
\Psi_\lambda(z) \vcentcolon= \prod_{\alpha \in \Diffk^{-1}}
\left( 1 - e\left(\sigma(\alpha) \hlf{z}{\lambda - \alpha\ell} \right)  \right).
\]
Here, $\lambda - \alpha\ell$ runs over finitely many orbits under the operation of $\Gamma_\ell$, and 
$\sigma(\mu)$ is a sign introduced to assure absolute convergence.
The product has divisor  $\HeegU_\infty(\lambda)$. However, because 
of the sign $\sigma(\alpha)$, it is not invariant under $\Gamma_\ell$. 
Instead, there is a non-trivial automorphy factor. 

This is actually a desirable situation: By calculating the automorphy factor, we are able to determine 
the Chern class of $\HeegU_\infty(\lambda)$ in the cohomology group $\Htwo(\Gamma_\ell, \Z)$ (see
sections \ref{subsec:localbp} and \ref{subsec:chern}). 
It turns out to be given by the image $[c_\lambda]$ of a bilinear form in the cohomology: 
\begin{equation}\label{eq:chern_intro}
\begin{gathered}
c_\lambda([h,t], [h', t']) = -  \Im\left[  \abs{\sqrtD}  F_\lambda(t,t') \right] \quad(\text{for}\quad [h,t],[h',t'] \in \Gamma_\ell), \\
\text{where} \quad F_\lambda(x,y) \vcentcolon= \hlf{x}{\lambda}\hlf{y}{\lambda} + \hlf{\lambda}{x}\hlf{y}{\lambda} \quad(x,y\in \Dltc\otimes_{\OK}\C).
\end{gathered}
\end{equation}
Through this, we know the Chern class of every local Heegner divisor 
as a finite linear combination, and can thus describe its position in the cohomology.

For this we use results prepared in section \ref{sec:cohom}, from calculations in the group cohomology for $\Gamma_\ell$,  
concerning the properties of cocycles in $\Htwo(\Gamma_\ell, \Z)$. 
We obtain an equivalent condition for the Chern class of a linear combination of Heegner divisors to be a torsion element, 
in Lemma \ref{lemma:tors_lcH}. From the proof, we also obtain a further, necessary condition, see Corollary \ref{cor:nec_cond}.
Finally, our main result, Theorem \ref{thm:H_torsCond}, describes exactly 
when Heegner divisors are torsion elements in the Picard group $\Pic(\Gamma_\ell\backslash U_\epsilon(\ell))$.
For a unimodular lattice $L$, the theorem can be formulated as follows, for the general version, see Theorem \ref{thm:H_torsCond} on p.\ \pageref{thm:H_torsCond}:
\begin{theorem}\label{thm:main_intro}
A finite linear combination of local Heegner divisors of the form
\[
\HeegU = \frac{1}{2} \sum_{\substack{m \in \Z \\ m<0}} c(m) 
\sum_{\substack{\lambda \in D \\ \hlf{\lambda}{\lambda} = m}} \HeegU_\infty(\lambda)
\]
with coefficients $c(m) \in \Z$, is a torsion element 
in the Picard group $\Pic(\Gamma_\ell\backslash U_\epsilon(\ell))$, if and only if
the equation 
\[
\sum_{\substack{m \in \Z \\ m<0}} c(m) \sum_{\substack{\lambda \in D \\ \Qf{\lambda} = m }}
\left[  F_\lambda(t,t') - \abs{\Disc}\frac{\hlf{\lambda}{\lambda}}{n} \hlf{t'}{t}   \right] = 0
\]
holds for all $t,t' \in \Dg$. Here, $F_\lambda$ is the bilinear form from \eqref{eq:chern_intro} above. 
\end{theorem}
As an application of the theorem, we study the obstructions for a (local) Heegner divisor to be a torsion element.
It turns out that they are given by certain spaces of cusp forms spanned by theta-series.
This result is Theorem \ref{thm:Obst} in section \ref{subsec:mf_lgkz}, which here can stated as follows,   
with $G=\SL_2(\Z)$ and $k=n+2$:
\begin{theorem}\label{thm:obst_intro}
A finite linear combination  of Heegner divisors
\[ 
\HeegU =  \frac{1}{2} \sum_{\substack{m \in \Z \\ m<0}} c(m) 
\sum_{\substack{\lambda \in D \\ \hlf{\lambda}{\lambda} = m}} \HeegU_\infty(\lambda)
\]
is a torsion element in $\Pic(\Gamma_\ell\backslash \Ueps(\ell))$ if and only if 
\[
\sum_{\substack{m \in \Z \\ m<0}} c(m) a(-m) = 0
\]
for all cusp forms $f \in \mathcal{S}_{k}^\Theta(G)$ with Fourier 
coefficients $a(m)$. Here, $\mathcal{S}_{k}^\Theta(G) \subset \mathcal{S}_{k}(G)$
denotes a space of cusp forms spanned by certain (positive-definite) theta-series, see p.\ \pageref{thm:Obst} for the precise definition. 
\end{theorem} 
Theorem  \ref{thm:obst_intro} can be seen a local analog to the global obstruction result showed by the author in \citep[][Section 5]{Hof14},
which in turn is a unitary group version of the obstruction theory  developed by Borcherds using Serre-duality \citep[see][Theorem 3.1]{Bo98}.   
We discuss the relationship between the local and the global obstruction theories in section \ref{subsec:localglobal},
and also how the two theorems relate to the quite similar results obtained by Bruinier and Freitag in the setting of orthogonal groups 
 \citep[see][Proposition 5.2, Theorem 5.4]{BrFr}.  Our results are also to some extent related to the results of Bruinier, Howard and Yang \cite{BHY15} 
and to recent work of Funke and Millson. 

\textit{The paper is structured as follows:}
In the first section, we present the set-up and notation used throughout.  
We introduce a Siegel domain model of the symmetric domain, with the
fixed isotropic lattice vector $\ell$ corresponding to the cusp at infinity.
We then describe the stabilizer of this cusp and define the Heisenberg group $\Gamma_\ell$.
Also, we sketch the construction of the compactification used for $X_\Gamma$.

In section \ref{sec:cohom}, we study the cohomology of the Heisenberg group $\Gamma_\ell$ and 
derive criteria describing when certain two-cocyles obtained from bilinear forms are torsion elements
in the cohomology  group $\Htwo(\Gamma_\ell, \Z)$. The following section \ref{sec:heegner} is the main part of the paper:
Here, we study Heegner divisors, we introduce the local Borcherds products and we determine their Chern classes.  
Using the results established in the second section, we get an equivalent condition for a linear combination of Heegner divisors to be  a torsion element in the cohomology,  Lemma \ref{lemma:tors_lcH} on p.\ \pageref{lemma:tors_lcH}. A further, necessary condition follows from the proof, see Corollary \ref{cor:nec_cond}. 
Finally, as our main result, we derive Theorem \ref{thm:H_torsCond}, part of which follows from the Lemma, while the converse is proved  constructively. 

The last section closes with the application to modular forms: 
In Theorem \ref{thm:Obst} we find that cusp forms arising from certain theta-series constitute the obstructions for a local Heegner divisor to be a torsion element in the Picard group.

\section{Hermitian lattices and symmetric domains}
\subsection{Hermitian spaces and lattices} \label{subsec:ltcs}
Let $\fieldk = \Q(\sqrt{\Disc})$ be an imaginary quadratic number field of discriminant $\Disc$, with $\Disc$ a square-free negative integer. 
Let $\OK \subset \fieldk$ be the ring of integers in $\fieldk$. Denote  by $\Diffk$ the different 
ideal and by $\Diffk^{-1}$ the inverse different ideal.

We shall consider $\fieldk$ as a subset of the complex numbers $\C$ and denote 
by $\sqrtD$ the square-root of the discriminant, with the usual choice of the complex square-root.
Then, $\Diffk$ is given by $\sqrtD \OK$ and $\Diffk^{-1}$ by $\sqrtD^{-1}\OK$.  

Let $V = \Vk$ be an indefinite  hermitian space over $\fieldk$ of signature $(1,n+1)$, 
endowed with a non-degenerate hermitian form denoted $\hlfempty$, linear in the left and conjugate linear in 
the right argument.
A complex hermitian space $V_\C = V\otimes_\fieldk\C$ is  obtained by extension of 
scalars. 
We denote by $V_\Q$ the $\Q$-vector space underlying $V$, which bears the structure of a quadratic 
space of signature $(2,2n+2)$  with the quadratic form $\Qf{\cdot}$ defined by $\Qf{x} \vcentcolon 
= \hlf{x}{x}$. Similarly, the real quadratic space underlying $V_\C$ is denoted $V_\R$. We have $V_\R =  V_\Q \otimes_\Q \R$.

Let $L$ be a lattice in $V$, with $L\otimes_{\OK}\fieldk = V$.
We denote by $L'$ the $\Z$-dual of $L$, defined as the set
\[
L' = \left\{ x\in V;\; \hlf{x}{y}\in \Diffk^{-1} \quad\text{for all $y \in L$}
\right\} = \left\{ x\in V;\; \tr_{\fieldk/\Q} \hlf{x}{y}\in \Z \quad\text{for all $y \in 
L$}\right\}.
\]
Naturally, $L'$ is a lattice in $V$, too. 
If $L\subseteq L'$, the lattice $L$ is called integral. If further for all $x\in L$, $\hlf{x}{x} \in \Z$, then $L$ is called even. Finally, $L$ is unimodular, if $L' = L$. 
The quotient $L'/L$ is referred to as the discriminant group of $L$.

More generally in the context of this paper, 
by a hermitian lattice  we mean a discrete subgroup $M$ of $V$, 
for which the ring of multipliers $\order{M}$ is an order in $k$.  
(A multiplier of $M$ is a complex number $\alpha$ with $\alpha M \subset M$.)
Most lattices will occur here as sub-lattices of a fixed lattice $L$,
with $L$ as above, of full rank, hermitian and even. 

Denote by $\Ug(V)$ the unitary group of $V$, and by $\SU(V)$ the special unitary group. The 
isometry group of a lattice $L$ in $\mathrm{U}(V)$ is denoted $\mathrm{U}(L)$, similarly for $\mathrm{SU}(L)$. The \emph{discriminant kernel} $\Gamma_L$ is the subgroup of finite index in $\SU(L)$ which acts trivially on the discriminant group of $L$.
We refer to subgroups of finite index in $\Gamma_L$ as unitary modular groups. 
In the following, $\Gamma$ will always denote a unitary modular group.

\subsection{A symmetric domain} \label{subsec:symm_domain}
Viewing $\Ug(V)$ as an algebraic group, its set of real points, denoted $\Ug(V)(\R)$, is the 
unitary group of $V_\C$. 
A symmetric domain for the action of $\Ug(V)(\R)$ on $V_\C$ is given by the quotient $\DomU = \Ug(V)(\R) / \mathcal{K}$ with a maximal compact subgroup $\mathcal{K}$.
Denote by  $\mathbb{P}V_\C$ the projective space of $V_\C$. A projective model for $\DomU$ is 
 given by the \emph{positive cone}
\[
\coneU = \left\{ [v]\in \mathbb{P}V_\C\,;\, \hlf{v}{v} > 0 \right\}. 
\]
We briefly review the construction of an affine model. 
Denote by $\Iso(V)$ the set of one-dimensional isotropic subspaces of $\Vk$. Its elements 
are in one-to-one correspondence with the rational boundary components of the symmetric domain. 
In particular, we fix an element $ I \in \Iso(V)$ by choosing a primitive isotropic 
lattice vector $\ell \in L$ and setting $I = \fieldk \ell$. Further, we choose a primitive vector 
$\ell' \in L'$ such that $\hlf{\ell}{\ell'} \neq 0$. We shall assume that $\ell'$ is isotropic, 
too. Note that this is a non-trivial assumption about the hermitian lattice $L$ and its 
dual. 

For $a \in V$, we denote by $a^\perp$ the orthogonal complement with respect to $\hlfempty$. We set 
$D\vcentcolon = L\cap \ell^\perp \cap \ell'^\perp$. Equipped with the restriction of $\hlfempty$, 
$D$ is a definite hermitian lattice of signature $(0,n)$. 
Denote by $W = W_\fieldk$ the subspace $D\otimes_{\OK}\fieldk$, and let $W_\C = W \otimes_\fieldk 
\C$. 

Now, an affine model for $\DomU$, called the Siegel domain model,
is given by the following generalized upper-half-plane:
\[
\HU_{\ell, \ell'} = \left\{ (\tau, \sigma) \in \C \times W_\C\,;\, 
2\Im(\tau)\abs{\sqrtD}\abs{\hlf{\ell}{\ell'}}^2 > 
- \hlf{\sigma}{\sigma}  \right\}. 
\]
For $(\tau, \sigma)\in\HU_{\ell, \ell'}$, we set 
\[
z = z(\tau, \sigma) \vcentcolon = \ell' - \sqrtD \tau\hlf{\ell'}{\ell}\ell + \sigma.
\] 
Clearly, under the canonical projection $\projV: V_\C \rightarrow \mathbb{P}V_\C$, we have
$\projV(z) \in \coneU$ for all $(\tau, \sigma)\in\HU_{\ell, \ell'}$. Conversely, every $[v] \in 
\coneU$ contains a representative of the form $z(\tau, \sigma)$ for some pair $(\tau, \sigma) 
\in \HU_{\ell, \ell'}$. Usually, in the following, since $\ell$ and $\ell'$ are fixed, we shall 
simply write $\HU = \HU_{\ell, \ell'}$.

The isotropic line $I_\C = I\otimes_\fieldk\C = [\ell]$ corresponds to the cusp at infinity of 
$\HU$.

\subsection{Stabilizer of the cusp} 
Next, we will describe the stabilizer in $\Gamma$ of the cusp $[\ell]$. 
Consider the following transformations corresponding to elements of $\SU(V)$: 
\begin{align}
\label{def:translationsU} [h,0]:\;& v \; \mapsto \; v - \hlf{v}{\ell}\sqrtD 
h \ell & \quad\text{for $h\in\Q$},    \\
\label{def:eichlerU} [0,t]:\;& v \; \mapsto \; v + \hlf{v}{\ell}t - 
\hlf{v}{t}\ell 
        - \frac{1}{2}\hlf{v}{\ell}\hlf{t}{t}\ell &
        \quad\text{for $t\in W$}.       
\end{align}
Clearly, these transformations stabilize the isotropic subspace $\fieldk\ell$. 
Their action on $\HU$ is given as follows:
\[
[h,0]: (\tau,\sigma) \mapsto (\tau + h,\sigma), \quad
[0,t]: (\tau,\sigma) \mapsto \left( \tau +
\frac{\hlf{\sigma}{t}}{\sqrtD\hlf{\ell'}{\ell}}
+\frac12 \frac{\hlf{t}{t}}{\sqrtD}, \sigma + \hlf{\ell'}{\ell}t\right).
\]
The \emph{Heisenberg group} attached to $\ell$, denoted $\HeisU{\ell}$, 
is the  set of pairs $[h,t]$ with group law given by
\begin{equation}\label{eq:Hgrp_law}
[h,t]\circ [h',t'] = \bigl[h + h' +
\frac{\Im\hlf{t'}{t}}{\abs{\sqrtD}}, t + t'\bigr].
\end{equation}
Here, we follow the convention that 
$\left([h,t]\circ[h',t']\right)v = [h,t]\left( [h',t']\, v \right)$ for $v\in V_\fieldk$. 
The center of the Heisenberg group consists of transformations of type \eqref{def:translationsU}. 

We denote  by $\Gamma_\ell$ the subgroup of $\Gamma$ given by the intersection $\Gamma\cap \HeisU{\ell}$,
its center we denote by $\Gamma_{\ell, T}$. 
The full stabilizer of the cusp in $\Gamma$  is given by the semi-direct product 
\[
\Gamma_\ell \ltimes (\Ug(W) \cap \Gamma)) = \operatorname{Stab}_\Gamma(\ell).
\]
Note that $\Gamma_\ell$ has finite index in the stabilizer. 

The elements of $\Gamma_\ell$ can be described as follows (this is well-known):
\begin{rmk}\label{rmk:ParabN}
Suppose $\Gamma$ is a unitary modular group and let $\Gamma_\ell = \Gamma \cap 
\HeisU{\ell}$.  Then there exist a positive rational number $\Ng$ and a lattice $\Dg$ of finite 
index in $\Dltc$, such that $[h,t] \in \Gamma_\ell$ for all $h \in \Ng\Z$, $t\in \Dg$, and that
$ \abs{\sqrtD}^{-1}\Im\hlf{t'}{t} \in \Ng\Z$ for all  $t,t'\in \Dg$.
\end{rmk}

\subsection{Boundary components}\label{subsec:cmpct}
The modular variety $X_\Gamma$ is given by the quotient
\[
\Gamma\backslash \DomU \simeq \Gamma \backslash \Ug(V)(\R) / \mathcal{K}
\simeq \Gamma \backslash \HU.
\]
Note that $X_\Gamma$ is non-compact. The usual Baily-Borel compactification  $X_{\Gamma, {BB}}^*$ is obtained by 
introducing a topology and a complex structure on the quotient
\[
 \Gamma  \backslash \left( \HU \cup \left\{I_\R;\; I \in \Iso(V)\right\} \right).
\]
We sketch this for the cusp at infinity, defined by $[\ell]$: 
The following sets constitute a system of neighborhoods of the cusp
\begin{equation} \label{eq:defUeps}
U_\epsilon(\ell) = \left\{ [z]\in \coneU ;\,  \frac{ \hlf{z}{z} }{ \abs{\hlf{z}{\ell}}^2 } 
\abs{\hlf{\ell'}{\ell}}^2> \frac1{\epsilon}  \right\} \quad\left( \epsilon>0 \right).
\end{equation}
A subset $V$ of $\coneU \cup \{[\ell]\}$ is called open if $V \cap  \coneU$ is open in 
the usual sense and further if $[\ell] \in V$ implies $U_\epsilon(\ell) \subset V$ for some $\epsilon>0$. 

Through the quotient topology, this construction yields a topology on $\Gamma\backslash \left( 
\coneU \cup \{ [\ell]\} \right)$. The complex structure is defined 
though the pullback under 
the canonical projection $\coneU \cap \{ I_\R;\; I \in \Iso(V)\} \rightarrow X_{\Gamma, 
{BB}}^*$, locally for each cusp, see \cite{Hof11} for details. 
This way, one gets the structure of a normal complex space on  $X_{\Gamma, {BB}}^*$. In general, 
however, there are still singularities at the boundary points. 

This difficulty can be avoided by using a toroidal compactification, instead. We recall the
construction briefly;  see \citep[][Chapter 1.1.5]{Hof11} and, in particular \citep[][Section 4.3]{BHY15} for more 
details. 
In the following, identify the sets $U_\epsilon(\ell) \subset \coneU$ with the 
corresponding sets of representatives in $\HU_{\ell, \ell'}$.
Clearly, the Heisenberg group $\Gamma_\ell$ operates on 
$U_\epsilon(\ell)$. For sufficiently small $\epsilon$, there is an open immersion 
\[
\Gamma_\ell \backslash  U_\epsilon(\ell) \rightarrow X_\Gamma. 
\]
Recall that for the center $C(\Gamma_ \ell) = \Gamma_{\ell, T}$, we have $\Gamma_{\ell, T} \simeq 
\Z\Ng$. Set $q_\ell \vcentcolon = \exp( 2\pi  i  \tau/ \Ng )$. 
The quotient $\Gamma_{\ell, T}\backslash U_\epsilon(\ell)$ 
can now be viewed as bundle of punctured disks over  $W_\C$: 
\[
V_\epsilon(\ell) \vcentcolon= \Gamma_{\ell, T}\backslash U_\epsilon(\ell) \simeq \left\{  (q_\ell, 
\sigma);\;
0<  \abs{q_\ell} < \exp\left(\frac{\pi \hlf{\sigma}{\sigma} + \epsilon^{-1}}{\abs{\sqrtD}^2 
\abs{\hlf{\ell'}{\ell}}^2} \right) \right\}.
\]
Adding the center to each disk, we get the disk bundle
\[
\widetilde{V_\epsilon}(\ell) \vcentcolon=  \left\{  (q_\ell, \sigma);  \;
 \abs{q_\ell} < \exp\left(\frac{\pi \hlf{\sigma}{\sigma} + \epsilon^{-1}}{\abs{\sqrtD}^2 
\abs{\hlf{\ell'}{\ell}}^2} \right) \right\}. 
\]
The action of $\Gamma_\ell$ is well-defined at each center, leaving the divisor $q=0$ fixed. 
 Also, if $\Gamma$ is sufficiently small, the operation is free, hence we get an open 
immersion
\begin{equation} \label{eq:immersionglue}
\Gamma_\ell \backslash U_\epsilon(\ell) \rightarrow \left( \Gamma_\ell / \Gamma_{\ell, T} 
\right) \backslash \widetilde{V_\epsilon}(\ell),
\end{equation}
by which the right hand side can be glued to $X_\Gamma$, yielding a partial 
compactification. 
For a point $(0, \sigma_0) \in \widetilde{V_\epsilon}(\ell)$, we define a system of open sets 
\[
B_\delta(0, \sigma_0) =  \left\{ (q_\ell, \sigma)\in\widetilde{V_\epsilon}(\ell) \,;\, 
\hlf{\sigma-\sigma_0}{\sigma - \sigma_0} < \delta, \abs{q_\ell} < \delta \right\}\qquad(\delta > 
0).
\]
Under the immersion \eqref{eq:immersionglue} the images of these sets form a system of open 
neighborhoods for the boundary point at $(0, \sigma_0)$. 

Repeating this construction and the gluing procedure for every $I \in \Gamma \backslash \Iso(V)$ yields a 
compactification of $X_\Gamma$, which we denote $X_{\Gamma, tor}^*$.

\section{The local cohomology group} \label{sec:cohom}
In the following, let $\Gamma$ be a unitary modular group and let $\Gamma_\ell \subset \Gamma$ be a 
Heisenberg group  of the form $\Gamma_\ell = \Ng\Z \rtimes \Dg$ with 
$\Ng\in\Q_{>0}$ and $\Dg\subseteq \Dltc$ as introduced in Remark \ref{rmk:ParabN}.
We are interested in the cohomology of $\Gamma_\ell$, more 
specifically the second cohomology group $\Htwo(\Gamma_\ell, \Z)$.

As usual, if $G$ is a group acting on an abelian group $A$, the n-th cohomology group is defined as 
the quotient
\[
 \KH^n(G, A) = \frac{\operatorname{ker}(\CoC^n(G, A)\xrightarrow{\partial} 
\CoC^{n+1}(G,A))}{\operatorname{im}(\CoC^{n-1}(G,A)\xrightarrow{\partial} 
\CoC^{n}(G,A))},
\]
wherein $\CoC^{n}$ is the set of $n$-cocycles, consisting of all functions $f: G^n\rightarrow A$, and $\partial$ is the coboundary operator. In the present setting, $G=\Gamma_\ell$, $A=\Z$ and the 
action  of $G$ is trivial.

Let $U_\epsilon(\ell)$ be a neighborhood of the cusp of infinity, as defined in 
\eqref{eq:defUeps} above, with $\epsilon$ sufficiently small, so that the map in \eqref{eq:immersionglue} 
is indeed an open immersion.
Further, denote by $\Oeps = \Oeps(U_\epsilon(\ell))$ the sheaf of holomorphic functions on $U_\epsilon(\ell)$ and by $\Oeps^* = \Oeps(U_\epsilon(\ell))^*$ the sheaf  of invertible holomophic functions. 
The action of $\Gamma_\ell$ on $U_\epsilon(\ell)$ naturally induces an action on $\Oeps$ and $\Oeps^*$. 
The exact sequence  
\[
\begin{tikzcd} 0 \arrow[r] & \Z \arrow[r, "\imath"] & \Oeps \arrow[r, "e"] 
& \Oeps^* \arrow[r] & 0 \end{tikzcd} 
\]
thus induces an exact sequence of cohomology groups:
\begin{equation}\label{eq:seqH2ZtoOps}
\begin{tikzcd} \KH^1(\Gamma_\ell, \Z) \arrow[r] & \KH^1(\Gamma_\ell, \Oeps) \arrow[r] & \KH^1(\Gamma_\ell, \Oeps^*) \arrow[r, "\delta"] & \KH^2(\Gamma_\ell, \Z) \arrow[r] 
& \KH^2(\Gamma_\ell, \Oeps). 
\end{tikzcd} 
\end{equation}
The Picard group of $\Gamma_\ell \backslash \Ueps(\ell)$ is given by  $\KH^1(\Gamma_\ell \backslash \Ueps(\ell), \Oeps^*)$. Since the open neighborhoods $\Ueps(\ell)$ are contractible, all analytic line bundles on $\Ueps(\ell)$ are trivial. Therefore,
\begin{equation}\label{eq:Pic_H1} 
\Pic(\Gamma_\ell \backslash \Ueps(\ell)) = \KH^1(\Gamma_\ell, \Oeps^*). 
\end{equation}
Further, let $\Peps$ denote the functions in $\Oeps$ which are periodic for the action of $\Ng\Z$. 
As $\Ng\Z = \Gamma_{\ell, T}$ is a normal subgroup with $\Gamma_\ell / \Ng\Z = \Dg$,  and since $\Ng\Z\backslash\Ueps(\ell)$ is contractible, we have
$\KH^p(\Gamma_\ell, \Oeps) = \KH^p(\Dg, \Peps)$ $(p=1,2,\dotsc)$.  
Thus, from the exact sequences in \eqref{eq:seqH2ZtoOps} and \eqref{eq:Pic_H1}, we get the exact sequence
\begin{equation}\label{eq:KHwithP} 
\begin{tikzcd} \frac{\displaystyle\Hom(\Dg, \Peps)}{\displaystyle\Hom(\Dg, \Z)} \arrow[r] & 
 \Pic( \Gamma_\ell \backslash \Ueps(\ell)) \arrow[r] &
 \KH^2(\Gamma_\ell, \Z) \arrow[r] 
& \KH^2(\Gamma_\ell, \Oeps).  
\end{tikzcd} 
\end{equation}
Further, since $\Dg$ is a free group, the following sequence is exact:
\[
\begin{tikzcd} 0 \arrow[r] & \Hom(\Dg, \Z) \arrow[r] & \Hom (\Dg, \Peps) \arrow[r] & \Hom(\Dg, \Peps^*) \arrow[r] & 0.
\end{tikzcd} 
\]
Whence,  from \eqref{eq:KHwithP} we find the exact sequence
\begin{equation}\label{eq:seqPictoH2}
\begin{tikzcd}
 \Hom(\Dg, \Peps^*) \arrow[r] & 
 \Pic( \Gamma_\ell \backslash \Ueps(\ell)) \arrow[r] &
 \KH^2(\Gamma_\ell, \Z) \arrow[r] 
& \KH^2(\Gamma_\ell, \Oeps). 
\end{tikzcd}
\end{equation}
Hence, to study $ \Pic( \Gamma_\ell \backslash \Ueps(\ell))$ we want to examine the structure of $\Htwo(\Gamma_\ell, \Z)$.

\subsection{Bilinear forms in the cohomology}
In this subsection, we will examine the image of certain bilinear forms in the cohomology. 
All calculations are carried out using the standard inhomogeneous complex of group 
cohomology \citep[cf.][Chapter 8]{Shimura}. 
\begin{definition}\label{def:BIL} 
Consider the set of bilinear forms $B: W_\C \otimes W_\C \longrightarrow \R$,
for which there is either a hermitian form $H$ or a symmetric complex bilinear form $G$ such that 
$B = \Im H$ or $B= \Im G$, respectively. 
Such forms generate a vector space of real bilinear forms on $W_\C$, which we denote $\BIL$.
Further, let $\BilZ$ denote  the set of forms in $\BIL$ which are $\Z$-valued on the lattice $\Dg$. 
\end{definition}
To a bilinear form in  $\BIL$ we can associate an element of $\Htwo(\Gamma_\ell, \Oeps)$:
Define the two-cocycle in $C^2(\Gamma_\ell, \Oeps)$ by setting 
 \begin{equation}\label{eq:BiltoH2e} 
B([h,t], [h',t']) \vcentcolon = B(t,t')\qquad\left( [h,t], [h',t'] \in \Gamma_\ell  \right).  
\end{equation}
The class $[B]$ of this cocycle is the image of $B$ in the cohomology. 
For $B \in \BilZ$ we also define a two-coycle in $\CoC^2(\Gamma_\ell,\Z)$ and the attached element in $\Htwo(\Gamma_\ell, \Z)$.
Thus, composing with the natural map $\Htwo(\Gamma_\ell, \Z)  \rightarrow  \Htwo(\Gamma_\ell, \Oeps)$ from \eqref{eq:seqH2ZtoOps} we have a sequence 
\begin{equation}\label{eq:seqBilZ}
\begin{tikzcd}
{\BilZ} \arrow{r} & \Htwo(\Gamma_\ell, \Z) \arrow{r} & \Htwo(\Gamma_\ell, \Oeps).
\end{tikzcd}
\end{equation}
The composition of the two maps in \eqref{eq:seqBilZ} is just the restriction to $\BilZ$ of the map  $\BIL \rightarrow \Htwo(\Gamma_\ell, \Oeps)$ defined by \eqref{eq:BiltoH2e}.
It turns out that the sequence  is exact:
\begin{proposition}\label{prop:seqExact}
The image of $\BIL$ in $\Htwo(\Gamma_\ell, \Oeps)$ vanishes.
\end{proposition}
\begin{proof}
In the following, let $B$ denote an element of $\BIL$. 
Clearly, it suffices to consider the following two cases: Either  $B$ arises from a hermitian form or $B$ arises from a bilinear form. 
\begin{enumerate}
\item Let $H: W_\C \times W_\C \longrightarrow \C$ be a hermitian form. 
Consider the following one-cocycle in $\CoC^1(\Gamma_\ell, \Oeps)$: 
\begin{gather*}
u( [h,t],z) = \frac1{2i}\left[ \frac{2}{\hlf{\ell'}{\ell}} H(\sigma, t) + H(t,t)\right]. \\ 
\shortintertext{Its image under the coboundary map it given by}
\begin{aligned}
\partial u([h,t],[h',t'], z)  &  = [h,t] u( [h', t'] , z)  - u([h,t][h',t'], z) + u([h,t])  \\
& = \frac1{2i}\left( 2 H(t,t') - H(t,t') - H(t',t) \right).
\end{aligned}
\end{gather*}
Thus, we see that $B = \Im H$ is indeed trivialized by a cochain. Hence, its image $\Htwo(\Gamma_\ell, \Oeps)$ vanishes
\item Let $G: W_\C \times W_\C \longrightarrow \C$ be a symmetric complex bilinear form. 
We consider the following one-cocycle valued in $\Oeps$:
\begin{gather*}
u( [h,t],z) =  \frac{i}{2}\left( \frac{1}{\hlf{\ell'}{\ell}} G(\sigma, t) + \frac12 \overline{G(t,t)}\right).\\
\shortintertext{Its image under the coboundary map is given by } 
\partial u([h,t],[h',t'], z) =  \frac{1}{2i} \left( G(t',t) - \frac12\overline{G(t,t')} - \frac12 \overline{G(t',t)}\right) =  \frac{1}{2i} \left( G(t',t) - \overline{G(t',t)}\right).
\end{gather*}
Thus $B = \Im G$ is trivialized by a cochain, and $[B] = 0$ in $\Htwo(\Gamma_\ell, \Oeps)$.
\end{enumerate}
\end{proof}
\begin{rmk} We note that under a map of the type defined in \eqref{eq:BiltoH2e}, the real parts of sesquilinear forms have vanishing image in $\Htwo(\Gamma_\ell, \Oeps)$, too. The proof is quite similar.
\end{rmk}
Now that we know the sequence \eqref{eq:seqBilZ} to be exact, we study the first map  $\BilZ \rightarrow \Htwo(\Gamma_\ell, \Z)$. 
It is far from being  injective. The following Lemma and its proof are essentially due to Freitag, 
a sketch is contained in \cite{FreitagNote07}. 
\begin{lemma}\label{lemma:torsCrit} 
The kernel of the map $\BilZ \rightarrow \Htwo(\Gamma_\ell, \Z)$  is the cyclic group generated by 
the antisymmetric bilinear form 
\[ 
\frac{1}{\Ng}\frac{\Im\hlfempty}{\abs{\sqrtD}}. 
\]
In particular, the image of an element $B \in \BilZ$ is a torsion element in $\Htwo(\Gamma_\ell, \Z)$ if and if $B$ and 
$\abs{\sqrtD}^{-1}\Im\hlfempty$ are linear dependent over $\Z$. 
\end{lemma}
\begin{proof} The proof uses a \emph{transgression map}, which we introduce next.  

First note that, since the action of $\Z$ is trivial,  the map  $\BilZ \rightarrow \Htwo(\Gamma_\ell, \Z)$ factors over 
 $\Htwo(\Dg, \Z) = \Htwo((\Gamma_\ell/ \Ng\Z), \Z)$. With Proposition \ref{prop:seqExact} we have:
 \[
 \begin{tikzcd}
 \BilZ \arrow{r}  & \Htwo(\Dg, \Z) \arrow{r} & \Htwo(\Gamma_\ell, \Z)  \arrow{r} & {0.}
 \end{tikzcd}
\]
Now, the transgression $\trg$ is defined as the map for which the sequence 
\begin{equation} \label{eq:deftrg}
\begin{tikzcd}
 {\Z} \arrow{r}{\trg} & { \Htwo(\Dg, \Z)  } \arrow{r}  &{\Htwo(\Gamma_\ell, \Z)} \arrow{r} & 0
\end{tikzcd}
\end{equation}
becomes exact.  Thus, the kernel 
of the map into $\Htwo(\Gamma_\ell, \Z)$ is generated by the image 
of the identity map $\mathbf{1}: \Z \rightarrow \Z$ under $\trg$. 
The image $\trg(\mathbf{1})$ is represented by a coboundary 
$(t,t') \mapsto (\partial  u)([\cdot, t],  [\cdot,t'])$, 
with a one-cochain $u: \Gamma_\ell \rightarrow \Z$, which has to satisfy two conditions:
\begin{enumerate}
 \item $u([\Ng\,h,0])  = h$ for all $h \in \Z$ and 
 \item $(\partial u)([h,t],[h',t'])$ does not depend on $h$ or $h'$.
\end{enumerate}
A suitable $u$ is obtained by setting $u([\Ng h, t]) \vcentcolon = h$. We get 
\[
(\partial u)([h,t],[h',t'])  = [h,t]\, u\left([h',t']\right) - u([h,t][h',t']) 
 + u([h,t])  =  -\frac{1}{ \Ng}\frac{\Im\hlf{t'}{t}}{\abs{\sqrtD}}.
\]  
Hence, $\trg(\mathbf{1})$ is represented by the cocycle 
\[
  (t,t') \longmapsto  \frac{1}{\Ng} \frac{\Im\hlf{t}{t'}}{\abs{\sqrtD}}, 
\]
any integer multiple of which is then contained in the kernel. Thus, for any $B \in \BilZ$ the image 
$[B]$ is a torsion element precisely if it is linear dependent to $\trg(\mathbf{1})$ over $\Z$.
\end{proof}
The linear dependence condition in the Lemma can more conveniently be formulated thus: If $B \in \BilZ$,
the image is a torsion element if and only if there is a rational number $Q$, such that for all $t,t' \in \Dg$, the following equation holds:
\begin{equation}\label{eq:lindep_cond}
 B(t,t') - Q \frac{\Im\hlf{t}{t'}}{\abs{\sqrtD}} = 0 .
\end{equation}
Since $\Dg$ has full rank in $W_\fieldk$, by linear extension, equivalently, the equation holds for all $t,t' \in W$;  similarly for all $t,t' \in W_\C$. 
As an example for this, we give an application to hermitian forms.  
\begin{rmk}
Let $H$ be a $\fieldk$-valued hermitian form on $W_\fieldk$, and assume that $\frac{1}{\abs{\sqrtD}}\Im H \in \BilZ$.
Further assume that $H$ is linear in its left argument (otherwise,  invert the sign in the equation below).
Then,  the map $H \mapsto \frac{1}{\abs{\sqrtD}}\Im H$ defines a torsion element in $\Htwo(\Gamma_\ell, \Z)$ if and only if
the following equation holds for all $t,t' \in \Dg$:
\begin{equation}\label{eq:extorsHerm}
 H(t,t') + \frac{\operatorname{Tr}H}{n} \Im\hlf{t}{t'} = 0,  
\end{equation}
where the trace $\operatorname{Tr}$ is taken over a normalized orthogonal basis for $\hlfempty$. 
\end{rmk}
\begin{proof}
Taking the imaginary part of \eqref{eq:extorsHerm}, we see that indeed, if the equation holds,
$\Im H$ is a rational multiple of $\Im\hlfempty$ and thus $\frac{1}{\abs{\sqrtD}}\Im H$ defines a torsion element by  the Lemma. 

Conversely, assume that the image  is a torsion element  in $\Htwo(\Gamma_\ell, \Z)$. 
Then, by the Lemma, the form has to be linear dependent to 
 $\frac{1}{\abs{\sqrtD}}\Im\hlfempty$ and satisfies an equation of the form \eqref{eq:lindep_cond}. 
 Since by linear extension, the equation holds for all $t,t' \in W_\C$, we may replace $t$ by a purely imaginary multiple. 
The resulting equation, equivalent to the first, is the following: 
 \[
  \Re H(t,t')  -  Q\cdot \Re \hlf{t}{t'}= 0, 
 \]
 valid for all $t,t' \in W_\C$. Whence by linear combination of the two equations, we find  $H(t, t') = Q\hlf{t}{t'}$ for all $t,t' \in W_\C$.
To determine the factor of proportionality $Q$, we take the trace.  With 
$\operatorname{Tr}\hlfempty\mid_{W_\C} = - n$, 
we get $Q = -\frac{1}{n}{\operatorname{Tr}H}$.
 \end{proof}

\section{Local Heegner divisors and Borcherds products} \label{sec:heegner}
Our main interest here is to study the contribution of Heegner divisors to the local Picard group. 
For this purpose, we will introduce local Borcherds products and, with their help, calculate the Chern classes of local Heegner divisors in $\Htwo(\Gamma_\ell, \Z)$.  
Then, we apply the cohomological results from section \ref{sec:cohom}.  

The local Picard group $\Pic(X_\Gamma, \ell)$ is defined as  the direct limit of the Picard groups on the regular loci (in the Baily-Borel compactification $X_{\Gamma, BB}^*$) of the open neighborhoods $\Ueps(\ell)$ of the cusp attached to $\ell$: 
\begin{equation}
\label{eq:defPicLoc}
\Pic(X_\Gamma, \ell) = \varinjlim \Pic\bigl( \Ueps^{reg}\bigr).
\end{equation}
We can describe this local Picard group through the direct system $\Pic(\Gamma_\ell \backslash \Ueps(\ell))$ up to torsion, as $\Gamma_\ell$ has finite index in the stabilizer of the cusp, $\operatorname{Stab}_\Gamma(\ell)$.  As the quotient $\operatorname{Stab}_\Gamma(\ell)/ \Gamma_\ell$ operates on the direct limit $\varinjlim \Pic(\Gamma_\ell \backslash \Ueps(\ell))$, for the invariant part, one has
\begin{equation}\label{eq:torEpstorLoc}
\Pic(X_\Gamma, \ell) \otimes \Q = \left(\varinjlim \Pic(\Gamma_\ell \backslash \Ueps(\ell)) \otimes \Q \right)^{\operatorname{Stab}_\Gamma(\ell) / \Gamma_\ell}. 
\end{equation}
Thus, to describe the position of a local divisor up to torsion, it suffices to work  with the Picard group  $\Pic(\Gamma_\ell \backslash \Ueps(\ell))$ for a fixed (sufficiently small) $\epsilon>0$. 
\begin{rmk}\label{rmk:toroidal_Pic}
Replacing the Baily-Borel compactification with the toroidal compactification $X_{\Gamma, tor}^*$, 
the system of open neighborhoods $\Ueps$ is replaced by  the system of open neighborhoods  $\widetilde{V_\epsilon}(\ell)$  with the operation of $\Gamma_\ell / \Gamma_{\ell, T}$, and one can look at the Picard groups $\Pic\bigl(\Gamma_\ell / \Gamma_{\ell, T}\backslash \widetilde{V_\epsilon}(\ell)\bigr)$. 
The main difference here is, that the divisor of $\{ q_\ell = 0 \}$ is now 
a non-trivial element of the Picard group. A function with this divisor is given by $q_\ell = e(\Ng^{-1}\tau)$. Note that the Chern class of $\{ q_\ell =0 \}$ is precisely $\Ng^{-1}\frac{\Im\hlf{t}{t'}}{\abs{\sqrtD}}$.  
\end{rmk}

\subsection{Local Heegner divisors}\label{subsec:heegner}
First, we recall the usual definition of Heegner divisors on $\HU$ \citep[cf.][Section 6]{Hof14},  
and introduce local Heegner divisors in the neighborhoods 
$U_\epsilon(\ell)$ of the cusp $[\ell]$.  

Let $\lambda\in L'$ be  a lattice vector of negative norm, i.e.\ $\hlf{\lambda}{\lambda}<0$. 
The (primitive) \emph{Heegner divisor} $\HeegU(\lambda)$ attached to $\lambda$ is a divisor on 
$\HU$ given by 
\[
\begin{aligned}
\HeegU(\lambda) & \vcentcolon =
 \left\{ (\tau, \sigma) \in \HU\,;\, \hlf{\lambda}{z(\tau,\sigma)} = 0 \right\}, \\
\end{aligned}
\] 
with $z(\tau, \sigma) = \ell' - \tau\sqrtD\hlf{\ell}{\ell'}\ell + \sigma$ (see section  
\ref{subsec:symm_domain}). 
Clearly, the divisor $\HeegU(\lambda)$ intersects 
$U_\epsilon(\ell)$ for every $\epsilon>0$, if and only if $\hlf{\lambda}{\ell} = 0$. 
In the following, we denote by $\ell^\perp$  the (orthogonal) complement of $\ell$ with respect to $\hlfempty$.

Thus, let $\lambda \in L' \cap  \ell^\perp$. Then, $\lambda = \lambda_\ell \ell + \lambda_D$ with $\lambda_D 
\in W_\fieldk$ and $\HeegU(\lambda)$ is given by an equation of the form 
\[
\lambda_\ell \hlf{\ell}{\ell'} + \hlf{\lambda_D}{\sigma} = 0. 
\]
Consider the orbit of $\lambda$ under $\Gamma_\ell$.
Since $\Gamma$ is a modular group, 
the Heisenberg group $\Gamma_\ell$ operates trivially on the discriminant group $L'/L$ and
thus $[h,t]\lambda \equiv \lambda\pmod{L}$ for all $[h,t] \in \Gamma_\ell$. 
Also, since $\lambda \in \ell^\perp$, it remains fixed under $[h,0]$  for all  $h\in\Ng\Z$,  and $\Gamma_{\ell, T}$ acts trivially. 

For an Eichler element $[0,t]$ with $t\in\Dg$, we have
\[
[0,t]\lambda = \lambda - \hlf{\lambda_D}{t}  \ell = 
(\lambda_\ell - \hlf{\lambda_D}{t}) \ell + \lambda_D.
\]
Thus, the orbit of $\lambda$ under $\Gamma_\ell/\Gamma_{\ell, T} \simeq \Dg$ is given by
$\lambda - \mathfrak{T}\ell$, where $\mathfrak{T}$ denotes the set 
\[
\mathfrak{T} = \mathfrak{T}(\lambda) \vcentcolon = 
\left\{ 
\hlf{\lambda}{t}\, ;\, t\in\Dg 
 \right\}. 
\]
Note that $\mathfrak{T} \subseteq \Diffk^{-1}$ (as a fractional ideal), since $\Dg \subseteq \Dltc$. 

Hence,  the group $\Gamma_\ell$ operates on the set $\lambda + \Diffk^{-1}\ell$ with only finitely many orbits and thus, the divisor
\begin{equation}\label{eq:defHinfty}
\HeegU_\infty(\lambda) \vcentcolon = \sum_{\alpha \in \Diffk^{-1}} 
\HeegU(\lambda + \alpha\ell)
\end{equation}
is invariant under $\Gamma_\ell$ and defines an element of 
$\operatorname{Div}{\left( \Gamma_\ell \backslash U_\epsilon(\ell)\right)}$.

\paragraph{Heegner divisors with index} 
Now, let $\beta \in L'/L$ be an element of the discriminant group and  
 $m$ a negative integer. Then, the Heegner divisor of index $(\beta, m)$, 
defined as the (locally finite) sum
\begin{equation}\label{eq:HeegUglob}
 \HeegU(\beta, m) = \sum_{\substack{\lambda \in L' \\ \QfNop(\lambda) = m \\ \lambda + 
L = \beta}} \HeegU(\lambda), 
\end{equation}
is a $\Gamma$-invariant divisor on $\HU$. Under the canonical projection 
$\HeegU(\beta, m)$ is the inverse image of the a divisor on $X_{\Gamma}$. 
Also note that $\HeegU(\beta, m) = \HeegU(-\beta,m)$. 

Through the open immersion $\Gamma_\ell \backslash U_\epsilon(\ell) \hookrightarrow 
\Gamma\backslash \HU = X_\Gamma$ from section \ref{subsec:cmpct}, the inclusion $U_\epsilon(\ell) 
\subset \HU$ and the projection maps, we get a commutative diagram
\[
\begin{tikzcd}
 {\operatorname{Div}(X_\Gamma)} \arrow{r} \arrow{d} & { \operatorname{Div}(\Gamma_\ell \backslash 
U_\epsilon(\ell)) } \arrow{d} \\
{\operatorname{Div}(\HU)} \arrow{r} & { \operatorname{Div}(U_\epsilon(\ell))\;. }  
\end{tikzcd}
\]
We denote by $\HeegU_\ell(\beta, m)$  the image in $\operatorname{Div}(\Gamma_\ell \backslash 
U_\epsilon(\ell))$ of the divisor $\HeegU(\beta, m) \in \operatorname{Div}(X_\Gamma)$. 
The corresponding $\Gamma_\ell$-invariant divisor in  $\operatorname{Div}(U_\epsilon(\ell))$ is 
also denoted by $\HeegU_\ell(\beta, m)$. 

For sufficiently small $\epsilon$, the  divisor $\HeegU_\ell(\beta, m)$ is given by the restriction 
 to $U_\epsilon(\ell)$ of the sum on the right hand-side of \eqref{eq:HeegUglob}. 
Then, only $\lambda$'s perpendicular to $\ell$ contribute. 
In particular, if $\HeegU_\ell(\beta, m)$ is non-zero, then $\beta$ is contained 
in the subgroup 
\[
\Lcal \vcentcolon = \left\{ 
\gamma  \in  L'/L\,;\; 2\Re\hlf{\gamma}{\ell} \equiv 0 
\bmod{M_1}\quad\text{and}\quad \abs{\sqrtD}\Im\hlf{\gamma}{\ell} \equiv 0 
\bmod{M_2} \right\} \subseteq L'/L, 
\]
where $M_1$, $M_2$ are the unique integers given by 
$2\Re\hlf{L}{\ell} = M_1 \Z$ and by $\abs{\sqrtD}\Im\hlf{L}{\ell} = 
M_2\Z$. 

With $\beta \in \Lcal$ the local divisor $\HeegU_\ell(\beta, m)$ can be written in the form
\begin{equation}\label{eq:HeegUell_bm}
\begin{aligned}
\HeegU_\ell(\beta, m) 
 = \sum_{\substack{ \kappa \in  \Dltc \\ \QfNop(\kappa + \dot\beta) = m}}\HeegU_\infty(\kappa + 
\dot\beta).  
\end{aligned}
\end{equation}
Here, we adopt the notation of \cite{BrFr} section 4, by which $\dot\beta$ denotes a 
representative of $\beta$ with $\dot\beta \in L' \cap \ell^\perp$, fixed once and for all for 
every $\beta \in \Lcal$. 
Note that a surjective homomorphism is given by 
\[
\pi: \Lcal \longrightarrow \Dltc'/\Dltc, \quad \beta \longmapsto \dot\beta_D,
\]
where $\dot\beta_D$ denotes the definite part of $\dot\beta$.

\subsection{Local Borcherds products} 
\label{subsec:localbp} 
In this section, our aim is to use local Borcherds products to describe the 
position of Heegner divisors in the cohomology. Given a lattice 
vector $\lambda$ of negative norm with $\lambda \in L'\cap \ell^\perp$ 
we can realize the local Heegner divisor attached to $\lambda$ through an infinite product with 
factors of the form $\left(1 - e\left(\hlf{z}{[0,t]\lambda}\right)\right)$ with $[0,t] \in \Gamma_{\ell,T}$.  

If for the Heegner divisor $\HeegU_\infty(\lambda)$ as in \eqref{eq:defHinfty}, we set
\[
\prod_{\beta \in \Diffk^{-1}}
\bigl[ 1 - e\left( \sigma(\beta) \hlf{z}{\lambda - \beta\ell} \right)\bigr], 
\qquad \text{with}\quad \sigma(\beta) \in \{\pm 1 \},
\]
we get an infinite product with (zero-)divisor $\HeegU(\infty)$. 
For $\sigma(\beta) \equiv 1$ the product would be  $\Gamma_\ell$-invariant. 
However, to assure absolute convergence, we must define  the sign $\sigma(\beta)$ depending on $\Im\beta$. 
Then, the product is no longer fully invariant. Instead, the operation of Eichler 
transformations gives rise to a non-trivial automorphy factor, which we will use to determine the 
position of $\HeegU_\infty(\lambda)$ in the local Picard group.
\begin{assumption*}From here on, we shall require that $\hlf{\ell}{\ell'} = \sqrtD^{-1}$.  
\end{assumption*}
We remark that this is not a particularly serious restriction, as under the assumptions concerning $\ell$ and $\ell'$ from section \ref{subsec:ltcs}, it is always possible to choose $\ell'$ suitably.

Now, keeping in mind that  $\Diffk = \sqrtD^{-1}\OK$ and $\overline{\mathcal{O}}_\fieldk = \OK = - \OK$, we define  the local Borcherds products  as follows:
\begin{definition}\label{def:localBp}
 Let $\lambda \in L'$ be a negative norm lattice vector in the orthogonal complement of $\ell$. The local 
Borcherds product $\Psi_\lambda(z)$ attached to $\HeegU_\infty(\lambda)$ is defined as
\[
\Psi_\lambda(z) \vcentcolon=  \prod_{\alpha \in \OK} 
\left[1 - e\left(\sigma(\Im\alpha)\left( 
\hlf{z}{\lambda} + \frac{\alpha}{\abs{\Disc}}  \right)
\right) \right],
\]
with a sign $\sigma(\Im\alpha)$ defined as follows:
\[
\sigma(\Im\alpha) = \begin{cases}
\phantom{-}1 & \quad \text{if $\Im \alpha \geq 0$,} \\
-1 & \quad\text{otherwise.}
\end{cases} 
\]
\end{definition}
Clearly, $\Psi_\lambda(z)$ is an absolutely convergent infinite product with divisor $\HeegU_\infty(\lambda)$.
With $\Diffk^{-1} = \sqrtD^{-1}\left( \Z + \zeta\Z\right)$, where $\Im \zeta = \frac12 \sqrtD$ and $2\Re\zeta \equiv \Disc \pmod{4}$, 
we can write $\Psi_\lambda(z)$ in the following form
\[
\Psi_\lambda(z)  =
\prod_{\substack{p\bmod{ \abs{\Disc}} \\ q \in \Z}}
\left[1 - e\left(\sigma(q)\left( 
\hlf{z}{\lambda} + \frac{1}{\abs{\Disc}} \left( p + \zeta\, q \right)\right)  
\right) \right],
\]
with $\sigma(q) = \operatorname{sign}(q)$ if $q\neq 0$ and $\sigma(0) = +1$. 

Note that $\Psi_\lambda$ is invariant under translations in $\Gamma_{\ell, T}$,
while the operation of Eichler transformations, $[0,t]$ with $t\in \Dg$,  gives rise to 
the (non-trivial) automorphy factor  
\begin{equation}\label{eq:J_def}
J_\lambda([h,t], z) = \frac{\Psi_\lambda([0,t]z)}{\Psi_\lambda(z)} \quad \left([h,t] \in 
\Gamma_\ell\right).
\end{equation}
\begin{proposition}\label{prop:J_lambda}
 The automorphy factor $J_\lambda$ attached to $\HeegU_\infty(\lambda)$ takes the form
 \[
J_\lambda([h,t], z) = 
e\left( -2\abs{\Disc}\hlf{z}{\lambda}\Re\hlf{t}{\lambda} - 2\left(\Re\hlf{t}{\lambda}\right)^2\zeta + \Re\hlf{t}{\lambda}\left( \zeta +1\right) \right),
 \]
with $\zeta$ such that $\OK = Z + \zeta\Z$.  Note that $J_\lambda$ is independent of the choice of $\zeta$.
 \end{proposition}  
\begin{proof}
Since $\hlf{\ell'}{\ell} =-{\sqrtD}^{-1}$, by \eqref{def:eichlerU}  we have 
$\hlf{[0,t]z}{\lambda} = \hlf{z}{\lambda} - {\sqrtD}^{-1} \hlf{t}{\lambda}$. 
Since $\hlf{t}{\lambda} = \hlf{t}{\lambda_D} \in \Diffk^{-1}$  we can write
\[
\hlf{[0,t]z}{\lambda} = \hlf{z}{\lambda}  + \frac{1}{\abs{\Disc}}(r +  \zeta s), \qquad\text{with}\; r,s\in\Z. 
\]
We note that $s = 2\Re\hlf{t}{\lambda}$.  
Now, after permuting representatives modulo $\abs{\Disc}$ 
and a shift in the index $q$, the automorphy factor from \eqref{eq:J_def} takes the form
\begin{equation}\label{eq:Jintermed}
   J_\lambda([h,t], z) = 
 \prod_{ p \bmod{\abs{\Disc}}}\prod_{q \in \Z}
 \frac{1 - e\left( \sigma(q-s)\left( 
 \hlf{z}{\lambda} + \abs{\Disc}^{-1} \left( p  + q \zeta \right) 
 \right)  \right)}%
 {1 - e\left( \sigma(q)\left( 
 \hlf{z}{\lambda} + \abs{\Disc}^{-1} \left( p  + q\zeta \right) \right)
 \right)}.
\end{equation}
Only factors with $\sigma(q-s) \neq \sigma(q)$ contribute to the 
 product. There are two cases:  Either we have $s > q \geq 0$, or $s \leq q < 0$. 
 We examine the first case. By applying the elementary identity 
 \[
\frac{1 - e(-z)}{1-e(z)}  = -e(-z).
 \]
 we get
\[
\begin{aligned}
J_\lambda([h,t], z)  
 & = \prod_{ p \bmod{\abs{\Disc}}}\prod_{0 \leq q < s}
 - e\left(  - \hlf{z}{\lambda} -  \frac{1}{\abs{\Disc}} \left( p  + q \zeta \right) 
  \right) \\
& = \prod_{ p \bmod{\abs{\Disc}}} (-1)^s
e\left( -s\hlf{z}{\lambda} - \frac{s}{\abs{\Disc}}\left( p +\frac{s-1}{2} \zeta\right)\right) \\
& = e\left( - s\abs{\Disc}\hlf{z}{\lambda} -   \frac{s(s-1)}{2} \zeta - \frac{s(\abs{\Disc} - 1)}{2}  + \frac{s\abs{\Disc}}{2} \right)\\
& = e\left( -s\abs{\Disc}\hlf{z}{\lambda} - \frac{s^2}{2}\zeta + \frac{s}{2}\zeta +\frac{s}{2} \right) \\
\end{aligned} 
\]
Hence, recalling that $s  = 2 \Re\hlf{t}{\lambda}$, we have
\begin{equation}\label{eq:thecocyclJ}
J_\lambda([h,t], z) = 
e\left(
-2\abs{\Disc}\hlf{z}{\lambda}\Re\hlf{t}{\lambda}
- 2\left(\Re\hlf{t}{\lambda}\right)^2\zeta 
+ \Re\hlf{t}{\lambda} \zeta + \Re\hlf{t}{\lambda} \right).
\end{equation}
We remark that the last term is  determined only up to sign, since $2\Re\hlf{t}{\lambda} \in \Z$.
Finally, we  note that as  the second term in \eqref{eq:thecocyclJ}
is a quarter-integer while $2\Re\zeta$ is only determined modulo $4$,
the automorphy factor is independent of the choice for $\Re\zeta$. 

The second case ($s\leq q < 0$) can be treated similarly, yielding the same result for the 
automorphy factor $J_\lambda([h,t], z)$. 
\end{proof}

\subsection{The Chern class of a Heegner divisor \texorpdfstring{$\HeegU_\infty(\lambda)$}{}}\label{subsec:chern}
From the automorphy factor $J_\lambda$ we now determine a two-cocycle representing the Chern class of 
the Heegner divisor $\HeegU_\infty(\lambda)$. 

\begin{proposition} \label{prop:ChernClassHl} 
The Chern class $\delta(\HeegU_\infty(\lambda))$ of the local Heegner divisor $\HeegU_\infty(\lambda)$ in $\Htwo(\Gamma_\ell, \Z)$ is determined by the cocycle
\[
[c_\lambda]: \quad ([h,t][h',t']) \longmapsto - 2\abs{\sqrtD} \Re\hlf{t}{\lambda}\Im\hlf{t'}{\lambda} 
= \Im\left( -\abs{\sqrtD} F_\lambda(t,t') \right),  
\]
where $F_\lambda(t,t') \vcentcolon = 2\Re\hlf{t}{\lambda} \hlf{t'}{\lambda}$.
\end{proposition}

\begin{proof}
To calculate the Chern class, we must realize the connecting homomorphism $\delta: \KH^{1}(\Gamma_\ell, \Oeps^*) \rightarrow \KH^{2}(\Gamma_\ell, \Z)$.
Thus, let $A(g,z)$ be a holomorphic function satisfying $ J_\lambda (g,z) = e\left(A(g,z)\right)$ and
set 
\begin{equation}\label{eq:infcocycl}
c(g,g') = A(gg', z) - A(g,g'z) - A(g', z) \quad\text{for all $g,g' \in \Gamma_\ell$}.
\end{equation}
Then, the two-cocycle defined by the map $(g,g) \mapsto c(g,g')$ is a representative for the Chern class in $\KH^2(\Gamma_\ell, \Z)$. 
Note that while $A(g,g')$ is not uniquely determined, $c(g,g')$ is independent of this choice; also, multiplying $J_\lambda$ with a trivial automorphy factor 
changes $c(g,g')$ only by a coboundary. 

Clearly, it suffices to  calculate $c(g,g')$ for Eichler transformations $g = [0,t]$ and $g' = [0,t']$. 
From \eqref{eq:infcocycl} we see that the last two terms in \eqref{eq:thecocyclJ}, being linear in $t$, cancel. We calculate 
\begin{align*}
 A([0, t+t'], z) &  - A([0,t],[0,t']z) - A([0, t'], z) =\\
  = &\; 2\abs\Disc \hlf{[0,t'] z - z}{\lambda} \Re\hlf{t}{\lambda} - 4\Re\hlf{t}{\lambda}\Re\hlf{t'}{\lambda}\zeta \\
 = &\; 2\sqrtD\Re\hlf{t}{\lambda} \hlf{t'}{\lambda} - 2\Re\hlf{t}{\lambda}\Re\hlf{t'}{\lambda} \sqrtD - 4\Re\hlf{t}{\lambda}\Re\hlf{t'}{\lambda}\Re\zeta \\
 = &\, -2\abs\sqrtD\Re\hlf{t}{\lambda}\Im\hlf{t'}{\lambda} - 4\Re\hlf{t}{\lambda}\Re\hlf{t'}{\lambda}\Re\zeta, 
\end{align*}
since $\Im\zeta =\frac12 \abs\sqrtD$.  
Now, consider the second term. We know $J_\lambda$ doesn't  depend on  the choice of $\Re\zeta$, thus
this term contributes at most a torsion element in the cohomology or vanishes entirely.  It can hence be ignored. 

Note also that the remaining first term is an integer for all $[0,t], [0,t'] \in \Gamma_{\ell, T}$.
\end{proof}
It is worth noting that the bilinear form $F_\lambda(\cdot, \cdot)$ introduced in Proposition \ref{prop:ChernClassHl} can be written in the form
\[
F_\lambda(a,b) = \Re\hlf{a}{\lambda} \hlf{b}{\lambda} = \hlf{b}{\lambda} \hlf{a}{\lambda} + \hlf{b}{\lambda} \hlf{\lambda}{a} \qquad(a,b \in W_\fieldk).
\]
Clearly, the first term is a complex bilinear form, while the second term is a hermitian form, we denote them by $B_\lambda(a,b)$ and $H_\lambda(a,b)$, respectively. Note that $H_\lambda(a,b)$ is linear in its \emph{second} argument. Further, we remark that $F_\lambda(a,b) = F_{\lambda_D}(a,b)$ for all $a, b \in W_\fieldk$.

\subsection{Torsion criteria for Heegner divisors}
Up to here, we have only worked on  Heegner divisors attached to individual lattice 
vectors, i.e.\ $\HeegU_\infty(\lambda)$, for $\lambda \in L'$ with 
$\Qf{\lambda}< 0$. 
Next, we consider linear combinations of Heegner divisors. We will be 
mainly interested in the Heegner divisors $\HeegU_\ell(\beta, m)$. 

For general linear combinations of  Heegner divisors, we have the following Lemma: 
\begin{lemma}\label{lemma:tors_lcH}  
Let $\HeegU$ be a finite linear combination of Heegner divisors of the form
 \[
 \HeegU = \sum_{\substack{\lambda \in L'\cap \ell^\perp \\ \QfNop(\lambda) < 0}} a(\lambda) \HeegU_\infty(\lambda), \qquad \left( a(\lambda) \in \Z \quad \text{for every $\lambda$}\right). 
 \]
Then, the Chern class $\delta(\HeegU)$ of $\HeegU$ is a a torsion element in $\Htwo(\Gamma_\ell, \Z)$
if and only if for all $t,t' \in \Dg$ the following equation holds
\[
 \sum_{\substack{\lambda \in L'\cap \ell^\perp \\ \QfNop(\lambda) < 0}} a(\lambda) 
 \left[ F_\lambda(t,t') - \frac{\hlf{\lambda}{\lambda}}{n} \hlf{t'}{t} \right] = 0.
\]
\end{lemma}
From the proof of this lemma we will also get the following necessary condition (where we use the same notation as in the Lemma):
\begin{corollary}\label{cor:nec_cond}
If $\delta(\HeegU)$ is a torsion element, then for the bilinear form $B_\lambda(a,b) = \hlf{a}{\lambda}\hlf{b}{\lambda}$ we have
\begin{equation}\label{eq:lemma_nec_cond}
\sum_{\substack{\lambda \in L'\cap \ell^\perp \\ \QfNop(\lambda) < 0}}
a(\lambda) \tr B_\lambda = 0, %
\end{equation}
where the trace  is taken over a normal orthogonal basis with respect to $\hlfempty$. 
\end{corollary}
\begin{proof}
The Chern class $\delta(\HeegU)$ is given by a linear combination 
 of cocycles $[c_\lambda]$ in $\Htwo(\Gamma_\ell, \Z)$.
 By  Proposition \ref{prop:ChernClassHl}, each $[c_\lambda]$ is represented 
 by the two-cocycle 
\[
(t,t') \mapsto -\Im\left[ \abs{\sqrtD} F_\lambda(t,t')\right]. 
\]
Through  \eqref{eq:Pic_H1} and the exactness of the sequence in \eqref{eq:seqH2ZtoOps}, the image of $[c_\lambda]$ in $\KH^2(\Gamma_\ell, \Oeps)$ vanishes.
By the results of section \ref{sec:cohom}, $\delta(\HeegU)$ is a torsion element in $\Htwo(\Gamma_\ell, \Z)$ if and only if there is a rational number $Q$ such that the equation 
\[
 \sum_{\substack{\lambda \in L'\cap \ell^\perp \\ \QfNop(\lambda) < 0}} a(\lambda)\abs{\sqrtD} \cdot\Im F_\lambda(t,t') 
 = Q \frac{\Im\hlf{t'}{t} }{\abs{\sqrtD}} 
\]
holds for all $t, t' \in \Dg$.  
Since $\Dg$ has full rank in $W_\fieldk$, by extension of scalars, the equation holds for all pairs of vectors in $W_\fieldk$. 
Both sides of the equation are linear in $t'$. Thus replacing $t'$ with a purely imaginary multiple gives a second, equivalent equation:
\[
\sum_{\substack{\lambda \in L'\cap \ell^\perp \\ \QfNop(\lambda) < 0}} a(\lambda)\abs{\sqrtD} \cdot \Re F_\lambda(t,t')
 = Q \frac{\Re\hlf{t'}{t} }{\abs{\sqrtD}}. 
 \]
By linear combination of the two equations, we get
\begin{equation}\label{eq:Q_cmplxcond} 
 \sum_{\substack{\lambda \in L'\cap \ell^\perp \\ \QfNop(\lambda) < 0}} a(\lambda)\abs{\sqrtD} 
 F_\lambda(t,t')  = Q  \frac{\hlf{t'}{t} }{\abs\sqrtD}. 
\end{equation}
To determine $Q$,  we take the trace of both sides of \eqref{eq:Q_cmplxcond}, using an orthogonal basis of $W_\C$ with respect to $\hlfempty$, say $\{e_l\}_{l=1,\dotsc, n}$ with $\hlf{e_l}{e_m} = -\delta_{l,m}$.
Now, $\tr \hlfempty = -n$ and the trace of $H_\lambda$ is $- \hlf{\lambda}{\lambda}$,
hence 
\begin{equation}\label{eq:factorQ} 
Q(-n) = \sum_{\substack{\lambda \in L'\cap \ell^\perp \\ \QfNop(\lambda) < 0}} a(\lambda) \abs{\Disc}
\left( - \hlf{\lambda}{\lambda} + \tr_{\{e_l\}} B_\lambda \right). 
\end{equation}
It turns out the trace of $B_\lambda$ does not contribute to $Q$. 
Indeed, if we take the trace of \eqref{eq:Q_cmplxcond} over an orthogonal basis of $W_\C$ obtained from $\{ e_l\}$ by rescaling with the complex unit $i$, i.e.\ $\{ i e_l\}_{l=1,\dotsc, n}$, 
the traces of the hermitian forms  $\hlfempty$ and $H_\lambda$ remain unchanged while that of $B_\lambda$ switches sign.
Comparing this result with \eqref{eq:factorQ}, we obtain
\[
Q = \sum_{\substack{\lambda \in L'\cap \ell^\perp \\ \QfNop(\lambda) < 0}}
a(\lambda)\cdot Q_\lambda
\quad\text{with}\quad 
Q_\lambda \vcentcolon= \abs{\Disc} \frac{\hlf{\lambda}{\lambda}}{n}.  
\]
Together with \eqref{eq:Q_cmplxcond} the statement follows. 
Further, since the contribution of $B_\lambda$ to the trace vanishes, we get the necessary condition
\begin{equation*}
\sum_{\substack{\lambda \in L'\cap \ell^\perp \\ \QfNop(\lambda) < 0}}
a(\lambda) \tr B_\lambda = 0. 
\end{equation*}
This proves the corollary, as well.
\end{proof}

\subsection{The Main Result}
We can now turn to the object of our main interest, Heegner divisors of the form $\HeegU_\ell(\beta, m)$.
We want to describe their position in the local Picard group. Recall that by \eqref{eq:HeegUell_bm} the divisors $\HeegU_\ell(\beta, m)$
can be written using divisors of the type $\HeegU_\infty(\lambda)$.  Thus, any finite linear combination $\HeegU$ of Heegner divisors $\HeegU_\ell(\beta, m)$, 
can be written as a locally finite sum of  Heegner divisors $\HeegU_\infty(\lambda)$. Also, note that  for a divisor of this type, the Chern
class  $\delta(\HeegU_\infty(\lambda))$ depends only on the projection $\lambda_D$. 
With this notation, we formulate the following theorem.
\begin{theorem}\label{thm:H_torsCond}
Consider a finite linear combination of local Heegner divisors of the form 
\begin{equation}
\HeegU = \frac12 \sum_{\beta \in \Lcal} \sum_{\substack{m \in \Z + \QfNop(\beta) \\  m <0 }}
c(\beta, m) \HeegU_\ell(\beta, m), 
\end{equation}
with integral coefficients $c(\beta, m)$, satisfying $c(\beta, m) = c(-\beta,m)$.

Then, $\HeegU$ is torsion element in the  Picard group $\Pic\left(\Gamma_\ell  \backslash\Ueps(\ell)\right)$ 
if and only if for all $t, t' \in \Dg$ the following equation holds
\begin{equation}\label{eq:tors_H}
\sum_{\beta \in \Lcal}\; \sum_{\substack{m \in \Z + \QfNop(\beta) \\  m < 0 }} c(\beta, m)  
\sum_{\substack{\lambda \in \Dltc' \\ \lambda + \Dltc \equiv\pi(\beta) \\ 
\QfNop(\lambda) = m}} \left[
F_\lambda(t,t') - \frac{\hlf{\lambda}{\lambda}}{n} \hlf{t'}{t}
\right] = 0.
\end{equation}
Further, a necessary conditions for this to be the case is 
that the following identity holds, with $B_\lambda(x,y) = \hlf{x}{\lambda}\hlf{y}{\lambda}$:
\begin{equation}\label{eq:tr_cond} 
\sum_{\beta \in \Lcal}\; \sum_{\substack{m \in \Z + \QfNop(\beta) \\  m < 0 }} c(\beta, m)  
\sum_{\substack{\lambda \in \Dltc' \\ \lambda + \Dltc \equiv\pi(\beta) \\ 
\QfNop(\lambda) = m}} \tr B_\lambda = 0.
\end{equation}
Here, the trace is taken over an orthogonal basis with respect to $\hlfempty$. 
\end{theorem}
We note that  by \eqref{eq:torEpstorLoc} a linear combination of Heegner divisors $\HeegU$ is a torsion element in $\Pic(\Gamma_\ell\backslash U_\epsilon(\ell))$ if and only if it is a torsion element in the local Picard group $\Pic( X_\Gamma, \ell)$.
\begin{proof}
\emph{If $H$ is a torsion element}, the equation \eqref{eq:tors_H} follows from Lemma \ref{lemma:tors_lcH}. 
Also, from the proof of that Lemma and Corollary \ref{cor:nec_cond}, it is clear that in this case, the identity 
\eqref{eq:tr_cond} holds.  

\emph{For the converse,} assume that \eqref{eq:tors_H} 
holds for all $t, t' \in \Dg$. We will show that $\HeegU$ is a torsion element in the Picard group.
By extension of scalars, the equation remains valid for all $t,t' \in W_\C$. 
Using \eqref{eq:thecocyclJ}, an automorphy factor describing $\HeegU$ in 
$\Pic\left(\Gamma_\ell\backslash\Ueps(\ell)\right)$ is given by the following (finite) product (for $g=[h,t] 
\in \Gamma_\ell$, $z\in\Ueps = \Ueps(\ell)$):
\begin{multline}\label{eq:JH_factors}
J_{\HeegU}(g, z)   = \prod_{\substack{\beta \in \Lcal \\ m \in \Z + \QfNop(\beta) \\ m<0}}
\prod_{\substack{\kappa \in \Dltc\\  \QfNop(\kappa + \dot\beta) = m }} 
  J_{\kappa + \dot\beta}(g, z)^{\frac12 c(\beta,m)} \\  
= \prod_{\beta, m}\prod_{\kappa}
e\Bigl(
-2\abs{\Disc}\hlfa[big]{z , \kappa + \dot\beta}\Re\hlfa[big]{t,\kappa + \dot\beta_\Dltc}  \\
- 2\zeta\left(\Re\hlfa[big]{t, \kappa + \dot\beta_\Dltc}\right)^2
+ \Re\hlfa[big]{t, \kappa + \dot\beta_\Dltc}\left( \zeta + 1\right) \Bigr)^{\frac{ c(\beta,m)}{2}}.
\end{multline}
Since $c(\beta,m) = c(-\beta,m)$, terms which are linear in the $\kappa + 
\dot\beta_\Dltc$ cancel. The remaining factors are of the form
  \[
   e\left(-2\abs{\Disc}\hlfa[big]{z , \kappa + \dot\beta}\Re\hlfa[big]{t,\kappa + \dot\beta_\Dltc}
   - 2\zeta\bigl(\Re\hlfa[big]{t, \kappa + \dot\beta_\Dltc}\bigr)^2
   \right)^{\frac{c(\beta,m)}{2}}. 
 \]
Now, we write 
 $\bigl\langle {z, \kappa + \dot\beta} \bigr\rangle = \bigl\langle {z, \kappa + 
\dot\beta_\Dltc}\bigr\rangle +    \bigl\langle {z , \dot\beta - \dot\beta_\Dltc}\bigr\rangle$. 
 Since $\hlfa{\dot\beta, \ell} = 0$, the second part depends only on the 
constant $\ell'$-component of $z$. 
We get
\[
 \left[ e\left( - 2\sqrtD\overline{\dot\beta_\ell} \Re\hlfa[big]{t, \lambda_D}\right)
e\left( -2 \abs{\Disc}\hlfa[big]{\sigma, \kappa + \dot\beta_\Dltc} \Re\hlfa[big]{t, \kappa + \dot\beta_\Dltc} 
-2 \zeta\bigl( \Re\hlfa[big]{t, \kappa + \dot\beta_\Dltc}\bigr)^2
\right)\right]^{\frac{c(\beta,m)}{2}}.
\]
We ignore the first factor for the time being and examine the second factor.  There, the first term in the exponential is 
$-2\abs\Disc F_{\kappa + \dot\beta_D}(t,\sigma)$ while the second term is equal to
$-2\zeta \Re F_{\kappa + \dot\beta_D}(t,t)$.
We apply \eqref{eq:tors_H} to both terms, and can rewrite this factor in the form
\begin{align}
e\left( 2\Disc\vphantom{\frac{\dot\beta_D}{n}} \right. & \left.\frac{\hlfa[big]{\kappa + \dot\beta_D, \kappa + \dot\beta_D}}{n}\Bigl[
\hlf{\sigma}{t} - \frac{\zeta}{\Disc}\hlf{t}{t} \Bigr] \right) \nonumber\\ \label{eq:somefactors}
&\qquad = e\left( 2\Disc\frac{\QfNop\bigl(\kappa + \dot\beta_D\bigr)}{n}\Bigl[
\hlf{\sigma}{t} - \frac{1}{2\sqrtD}\Qf{t} \Bigr] \right) e\left(  - 2\Re\zeta\cdot \frac{\Qf{\lambda}}{n} \Qf{t} \right). 
 \end{align}
Clearly, the last factor in \eqref{eq:somefactors} has finite order and is a torsion element in $\Pic(\Gamma_\ell\backslash\Ueps)$.
Now, we claim that the first factor is actually a trivial automorphy factor.  To see this, consider the invertible function 
$f(z) = e(c\tau)$ with $c\in\Q^\times$; under the operation of $\Gamma_{\ell, T}$, it gives rise to  the following trivial automorphy factor
\[
j_1([0,t], z) = \frac{f([h,t]z)}{f(z)} = e\left(c\left( - \hlf{\sigma}{t} + 
\frac{1}{2\sqrtD}\hlf{t}{t}\right)\right).
\]
Hence the first factor in \eqref{eq:somefactors} is indeed  trivial. 
Now, we return to the previously excluded factor 
\begin{equation}\label{eq:furtherfactors} 
e\left( - 2\sqrtD\overline{\dot\beta_\ell} \Re\hlfa[big]{t, \kappa + \dot\beta_D}\right) 
 =  e\left( - \left( \abs{\sqrtD} \Im\dot\beta_\ell +  \sqrtD\Re\dot\beta_\ell\right) 2\Re\hlfa[big]{t, \kappa + \dot\beta_D} \right). 
\end{equation}
Since $\abs{\sqrtD} \Im\dot\beta_\ell$  is rational (actually, half-integer), this term contributes only a torsion element in the Picard-group. 
Consider the invertible function  $g(z) = e\left(\hlf{\sigma}{\mu}\right)$  with $\mu \in W_\fieldk$ from which we get the trivial automorphy factor
\begin{equation*}
j_2(z,[0,t]) = e\left( \sqrtD^{-1} \hlf{t}{\mu}\right).
\end{equation*}
Setting $\mu = \kappa + \dot\beta_D$, we multiply \eqref{eq:furtherfactors} with a suitable power of $j_2$ 
to kill the term in $\Re\hlfa[big]{t, \kappa + \dot\beta_D}$. Then, only torsion elements remain, as
$\abs{\sqrtD} \Im\hlfa[big]{t, \kappa + \dot\beta_D}$ and $\Re\dot\beta_\ell$ are rational numbers.  

Thus,  we find that each of the finitely may factors of $J_\HeegU$ from \eqref{eq:JH_factors} 
can be expressed through suitable powers of  trivial automorphy factors of the types $j_1$ and $j_2$ and factors of finite order. 
Hence, it follows that $\HeegU$ is a torsion element in 
$\Pic\left(\Gamma_\ell\backslash \Ueps\right)$.
\end{proof}

\begin{rmk}
As in Remark  \ref{rmk:toroidal_Pic}, if one looks at the neighborhoods $\widetilde{V_\epsilon}(\ell)$ from section \ref{subsec:cmpct} rather than $U_\epsilon(\ell)$, 
in the proof,  the function $f$ is no longer invertible and the automorphy factor $j_1$ becomes non-trivial, 
since  $f$ vanishes on the disk center $\{ q_\ell = 0 \}$. As mentioned before, the Chern class is given by $(t,t') \mapsto (\Ng\abs{\sqrtD})^{-1}\Im\hlf{t}{t'}$.  

Thus, in the Theorem one would have to replace ''torsion element in $\Pic(\Gamma_\ell \backslash \Ueps(\ell))$`` by ''equivalent to the divisor of  $\{ q_\ell = 0 \}$ in 
$\Pic(\Gamma_\ell/\Gamma_{\ell, T}\backslash \widetilde{V_\epsilon}(\ell)\bigr)$ (up to torsion)``.  
This kind of statement also carries over to the direct limit and describes the position (up to torsion) of $\HeegU$ in
$\varinjlim \Pic(\widetilde{V_\epsilon}(\ell))$, which in analogy to \eqref{eq:defPicLoc} may be considered as a local Picard group for the cusp $[\ell]$ on $X_{\Gamma, tor}^*$. 
\end{rmk}

\begin{rmk} \label{rmk:BrFr_thm}
Recall how the rational space  $V_\Q$ underlying $V_\fieldk$ has the structure of a quadratic space of signature $(2,2n+2)$. 
Let $\Og(V)$ be the orthogonal group of  $V_\Q$ and $\Og(V)(\R)$ its set of real points. 
In \cite{BrFr}, Bruinier and Freitag study local Heegner divisors at generic boundary components of the symmetric domain 
for such indefinite orthogonal groups.  The local Heegner divisors we consider here can be described as the restriction of their local Heegner 
divisors:

For $\lambda \in \Dltc'$, $\HeegU_\infty(\lambda)$ is 
the restriction of a local Heegner divisor attached to $\lambda$ and, similarly, $\HeegU_\ell(\beta,m)$ is 
the restriction of a composite local Heegner divisor, in the local Picard group for a generic boundary component of the symmetric domain, 
defined by $\fieldk\ell$ as a two-dimensional isotropic subspace over $\Q$. 
This follows from the embedding theory developed by the author in \cite{Hof11, Hof14}. 

The relationship between Theorem \ref{thm:H_torsCond}  and the results in \cite{BrFr} is the following:
By  taking the real part of both sides of \eqref{eq:tors_H}, one gets precisely the torsion 
condition from \citep[Theorem 4.5]{BrFr}. It follows that, under these assumptions,
if a local Heegner divisor  $\HeegU$ as in Theorem \ref{thm:H_torsCond} is a torsion element in $\Pic\left(\Gamma_\ell\backslash\Ueps(\ell)\right)$,  there is a pre-image under restriction which satisfies the torsion criterion in \cite{BrFr} and hence is  a torsion element in the local Picard group for a generic boundary component of the  orthogonal modular variety.  
Conversely, for every local Heegner divisor there which restricts to $\HeegU$, the criterion of \cite{BrFr} implies that  that \eqref{eq:tors_H} holds for $\HeegU$.
\end{rmk}

\section{Application to modular forms}\label{subsec:mf_lgkz}
In this section as an application of Theorem \ref{thm:H_torsCond} we derive a 
statement describing obstructions to local  Borcherds products through certain vector valued cusp forms.  
Our results are closely related to those obtained by Bruinier and Freitag in the context of orthogonal groups \citep[see][Section 5]{BrFr}.

Let us briefly recall some standard facts about the Weil representation and definition of vector valued modular forms. 
The rational space $W_\Q$  underlying $W_\fieldk$, equipped with the quadratic form $\Qf{\cdot}$,  is negative definite with dimension $2n$, 
and the definite lattice $D$ it contains has even $\Z$-rank $2n$. 
Hence,  the Weil representation of the metaplectic group $\Mp_2(\Z)$,
defined as the pre-image of $\SL_2(\Z)$ under the double covering map $\Mp_2(\R)  \twoheadrightarrow  \SL_2(\R)$, factors over $\SL_2(\Z)$.

Thus, there is a unitary representation of $\SL_2(\Z)$ on the group algebra $\C[D'/D]$, denoted $\rho_D$. 
The dual representation to $\rho_D$ is denoted by $\rho_D^*$. 

The Weil-representation $\rho_D$ is defined through the action of the generators of $\SL_2(\Z)$,
$T = \left(\begin{smallmatrix} 1 & 1 \\ 0 & 1\end{smallmatrix} \right)$ 
and $S =\left(  \begin{smallmatrix} 1 & 1 \\ 0 & 1\end{smallmatrix} \right)$. 
Note that $\rho_D^*$ can be obtained from $\rho_D$ by complex conjugation of the matrix coefficients. Thus, we have
(cf.\ \cite{Shin75}):
\begin{equation*}
\begin{split}
\rho_D^*(T)\,\ebase_\gamma & = e\bigl(-\Qf{\gamma} \bigr)\ebase_\gamma, \\
\rho_D^*(S)\,\ebase_\gamma & = \frac{\sqrt{i}^{-2n}}{\sqrt{\abs{\Dltc'/\Dltc}}} 
\sum_{\delta \in \Dltc'/\Dltc}  e\bigl( \blf{\gamma}{\delta} \bigr) \ebase_\delta,
\end{split}
\end{equation*}
where $(\ebase_\gamma)_{\gamma\in \Dltc'/\Dltc}$ is the standard basis for the group algebra $\C[\Dltc'/\Dltc]$, 
and $\blfempty$  is the bilinear form on $V_\Q$ given by $\blfempty\vcentcolon= \tr_{\fieldk/\Q}\hlfempty$. 

\begin{definition}
For $k\in\Z$, a function $f:\, \Hp \rightarrow \C[\Dltc'/\Dltc]$ is called a vector valued modular form of weight $k$ with respect to 
$\rho_D^*$ if 
\begin{enumerate}
 \item $f( A\tau) = \left(c\tau + d\right)^{k} \rho_D^*(A) f(\tau)$ for all $A = \left( \begin{smallmatrix} a & b \\ c & d \end{smallmatrix}\right)\in\SL_2(\Z)$. 
\item $f$ is holomorphic on $\Hp$,
\item $f$ is holomorphic at the cusp $i\infty$. 
\end{enumerate}
\end{definition}
Here $\SL_2(\Z)$ acts on $\Hp$ as usual.
Thus, the first condition implies the existence of a Fourier expansion:
\[
f(\tau) = \sum_{\gamma \in \Dltc'/\Dltc} \sum_{m\in \Z - \QfNop(\gamma)}
a(\gamma, m) e(m\tau) \ebase_\gamma. 
\]   
The second condition means that all coefficients with $m<0$ vanish.
If  $a(\gamma, m)=0$ for all $m\leq 0$, then $f$ is called a \emph{cusp form}. 
We denote the space of cusp forms of weight $k$ transforming under $\rho_D^*$ by $\Cuspf_k(\rho_D^*)$.

In the following, set $k = n +2$. We will define certain $\C[\Dltc'/\Dltc]$-valued cusp forms in $\Cuspf_k(\rho_D^*)$ 
using theta-series with harmonic polynomials as coefficients:
\begin{equation}\label{eq:type_theta}
\Theta_p(\tau, v) = \sum_{\lambda \in \Dltc'} p(\lambda, v) e( - \Qf{\lambda}\tau)  
\ebase_\lambda,
\end{equation}
for fixed $v \in W_\C$. If $p(\lambda, v)$ is harmonic in $\lambda$ and homogeneous of degree two, the theta series is a cusp form in $\Cuspf_k(\rho_D^*)$. 
This is a well-known result in theory of theta functions which can be proved through Poisson-summation, see e.g.\ \citep[Theorem 4.1]{Bo98}. 

The polynomials in question are obtained from the torsion condition in our main result,  Theorem \ref{thm:H_torsCond}. This will allow us to identify
a space of cusp forms as the set of obstructions against the local Heegner $\HeegU$ being torsion. 

We rewrite \eqref{eq:tors_H} using polynomials $p_1(u,v,w), p_2(u,v,w)\in \R[u,v,w]$ defined as follows:
\begin{gather}\label{eq:tors_Hpoly}
\sum_{\beta \in \Lcal}\; \sum_{\substack{m \in \Z + \QfNop(\beta) \\  m < 0 }} c(\beta, m)  
\sum_{\substack{\lambda \in \Dltc' \\ \lambda + \Dltc \equiv\pi(\beta) \\ 
\QfNop(\lambda) = m}} \left[ p_1(\lambda, t, t') + i p_2(\lambda, t, t') \right] = 0, \\
\nonumber
\begin{aligned} 
\qquad \text{with}\quad &
p_1(u,v,w) \vcentcolon  = & \Re F_u(v,w) - \frac{\Qf{u}}{n} \Re\hlf{v}{w},  \\ 
 & p_2(u,v,w) \vcentcolon  = & \Im F_u(v,w) - \frac{\Qf{u}}{n} \Im\hlf{v}{w}.
\end{aligned}
\end{gather}
We note that for the real part of \eqref{eq:tors_H} to hold, by linearity, it suffices to verify for $t=t'$. Consequently, we set 
\[
P(u,v) \vcentcolon = p_1(u,v,v) = 2\left(\Re\hlf{u}{v}\right)^2 - \frac{\Qf{u}}{n} \Qf{v}. 
\]
It is easily seen that both $p_1$ and $p_2$ can be obtained from $P$ using the polarization identity, for example
\[
p_2(u,v,w) = \frac12\left( P(u,v) + P(u, -iw) - P(u, v - iw) \right). 
\]
We also note that these polynomials are all harmonic and homogeneous in $u$. In fact, $P$ is harmonic in both indeterminates $u$ and $v$ and also homogeneous of the correct degree.
Thus, in particular, for every $v\in W_\C$, the theta series $\Theta_P(\tau, v)$ is a cusp form transforming under $\rho_D^*$ with the desired weight $k$. We rewrite \eqref{eq:type_theta} slightly to obtain the Fourier expansion of $\Theta_P(\tau, v)$:
\[
\Theta_P(\tau, v) = \sum_{\gamma \in \Dltc'/\Dltc} \sum_{\substack{ m\in \Z - 
\Qf{\gamma} \\ m<0}} \biggl(
\sum_{\substack{\lambda \in \Dltc' \\ \lambda + \Dltc 
\equiv\gamma \\ \QfNop(\lambda) = m}} P(\lambda,v) \biggr)
\cdot e( -m\tau) \ebase_\gamma. 
\]
Now, the Fourier coefficients are precisely the real part of the inner sums in \eqref{eq:tors_Hpoly}, 
restricted to the diagonal with $t=t'=v$. 

As $v$ varies over $W_\C$, these theta-series $\Theta_P(\tau,v)$ span a subspace of $\Cuspf_k(\rho_D^*)$ which we denote as $\Cuspf_k^\Theta(\rho_D^*)$. 
We remark that the polynomials $p_1(\lambda, v,w)$ and $p_2(\lambda; v, w)$ also define theta series, but these are already contained in $\Cuspf_k^\Theta(\rho_D^*)$. 

With these considerations, Theorem \ref{thm:H_torsCond} can be restated using modular forms. 
\begin{theorem}\label{thm:Obst} 
A  finite linear combination of Heegner divisors 
\[
\HeegU = \frac12 
\sum_{\beta \in \Lcal}\sum_{\substack{m \in \Z + \QfNop(\beta) \\ m <0 }} 
c(\beta, m) \HeegU_\ell(\beta, m), 
\]
with integer coefficients $c(\beta, m)$ satisfying $c(\beta, m) = c(\beta,-m)$ is a torsion element in the 
 Picard group $\Pic\left(\Gamma_\ell\backslash\Ueps(\ell)\right)$ if and only if
\[
\sum_{\beta \in \Lcal} \sum_{\substack{m \in \Z + \Qf{\beta} \\ m<0}} c(\beta, 
m) a(\pi(\beta), -m) = 0
\]
for every cusp form $f = \sum_{\gamma \in \Dltc'/\Dltc} \sum_{m\in \Z - \QfNop(\gamma)}  a(\gamma, m) e( - m\tau) \ebase_\gamma \in 
\Cuspf_k^\Theta(\rho_D^*)$.
\end{theorem}

\subsection{Relationship to global obstruction theory and the work of Bruinier  and Freitag} \label{subsec:localglobal}

Since the statement of Theorem \ref{thm:Obst} holds for all sufficiently small $\epsilon$, passing to the direct limit we get the statement for the local Picard group at the cusp $\ell$.  Now Theorem \ref{thm:Obst} formally resembles a global obstruction statement for unitary groups from \cite{Hof14} in the style of Borcherds \cite{Bo99}. 
It can be states as follows (by \cite[Lemma 5, Theorem 4]{Hof14}): 
\begin{theorem}[From \cite{Hof14}] \label{thm:obstglob}
A Heegner divisor of the form 
\[
\HeegU = \frac12 \sum_{\beta \in L'/L}\sum_{\substack{m \in \Z + \QfNop(\beta) \\ m <0 }}  c(\beta,m) \HeegU(\beta, m)
\]
is the divisor of a Borcherds product if and only if 
\[
\sum_{\beta \in L'/L} \sum_{\substack{m \in \Z + \QfNop({\beta}) \\ m<0}}
c(\beta, m) b(\beta, -m) = 0,
\]
for every cusp form $g\in \Cuspf_k(\rho_L^*)$ with Fourier coefficients $b(\beta, -m)$. 
\end{theorem}
Since by results of Bruinier \cite{Br02}, the local obstruction space $\Cuspf_k(\rho_\Dltc^*)$  can be embedded into the global obstruction space $\Cuspf_k(\rho_L^*)$  \citep[see also][Section 5]{BrFr}, the global obstruction equation implies the local one. Thus, if $\HeegU$ is the divisor of a Borcherds product, then the local divisor 
\[
\HeegU_\ell \vcentcolon = \frac12 \sum_{\beta \in \Lcal}\sum_{\substack{m \in \Z + \QfNop(\beta) \\ m <0 }}  c(\beta,m) \HeegU_\ell (\beta, m)
\]
is a  torsion element in the local Picard group. In fact, the same argument applies for every cusp. 

The results of Bruinier and Freitag, \citep[][Theorem 5.1]{BrFr}, in the setting of orthogonal groups are very similar to Theorem \ref{thm:Obst} above. 
The definition of their theta-series is (essentially) the same.  Indeed, if we look at the rational quadratic space $V_\Q$ underlying $V_\fieldk$ 
and the lattices $L$  and $\Dltc$ as quadratic modules in $V_\Q$, the obstruction spaces are the same. 

In this case, through the embedding theory from  \cite{Hof11, Hof14} we can pull back Heegner divisors on the modular variety of the orthogonal group to Heegner divisors for the modular variety of the unitary group. As sketched in Remark \ref{rmk:BrFr_thm} above, this also works locally. 
Thus if $H$ is the Heegner divisor of a Borcherds product for the orthogonal group, then it is \emph{trivial at generic boundary components} in the sense of \citep[][Definition 5.3]{BrFr}, i.e.\ locally torsion, and hence restricts to a torsion element in the local Picard group for every cusp $I$ of the unitary modular variety. Similarly, by pulling back the Borcherds product itself, one gets a Borcherds product for the unitary group with the pull-back of $H$ as its divisor, and through Theorem \ref{thm:obstglob}, again, the corresponding local Heegner divisors are torsion elements.  

We also remark that the obstruction space in Theorem \ref{thm:obstglob} is the same as that from Borcherds' \cite{Bo99} in the orthogonal situation. Hence, if $\HeegU$ is the Heegner divisor of a Borcherds product for the unitary group, one can find a Heegner divisor on the orthogonal side which restricts to $\HeegU$ and is the divisor of a Borcherds product.
\begin{rmk} In \cite[][Theorem 5.4]{BrFr},  they were able to show that for a unimodular lattice $L$, the
triviality of a Heegner divisor at generic boundary components, conversely, implies the global obstruction equation of Borcherds from \cite{Bo99} and hence the existence of a  Borcherds product for a Heegner divisor  that kills all local obstructions. 
Their argument depends on two results: The uniqueness of isomorphism classes of unimodular lattice, and a result of Waldspurger \citep[see][]{Wald78} on the generation of the space $\Cuspf_k(\rho_L^*)$ by theta-series for definite lattices.  
Unfortunately, there is no obvious way to transfer this argument to hermitian lattices, since given a quadratic module over $\Z$ a complex structure need neither exist, nor need it be unique. 
\end{rmk}

\acknowledegements I wish to thank Jan Bruinier for many helpful discussions and for his encouragement in the course of working on this paper.
Further, I would like to acknowledge the role of Prof. Freitag: His insights into group cohomology have helped me considerably to start out with this project. 
Finally, I want to thank the anonymous referee, whose comments have led to enormous improvements in this paper. 

\bibliographystyle{hplain} 
\bibliography{local.bib} 

\begin{thebibliography}{10}

\bibitem{Bo98}
Richard~E. Borcherds.
\newblock Automorphic forms with singularities on {G}rassmannians.
\newblock {\em Invent. Math.}, 132(3):491--562, 1998.

\bibitem{Bo99}
Richard~E. Borcherds.
\newblock The {G}ross-{K}ohnen-{Z}agier theorem in higher dimensions.
\newblock {\em Duke Math. J.}, 97(2):219--233, 1999.

\bibitem{Br02}
Jan~H. Bruinier.
\newblock {\em Borcherds products on {O}(2, {$l$}) and {C}hern classes of
  {H}eegner divisors}, volume 1780 of {\em Lecture Notes in Mathematics}.
\newblock Springer-Verlag, Berlin, 2002.

\bibitem{BrFr}
Jan~H. Bruinier and Eberhard Freitag.
\newblock Local {B}orcherds products.
\newblock {\em Ann. Inst. Fourier (Grenoble)}, 51(1):1--26, 2001.

\bibitem{BHY15}
Jan~Hendrik Bruinier, Benjamin Howard, and Tonghai Yang.
\newblock Heights of {K}udla-{R}apoport divisors and derivatives of
  {$L$}-functions.
\newblock {\em Invent. Math.}, 201(1):1--95, 2015.

\bibitem{B123}
Jan~Hendrik Bruinier, Gerard van~der Geer, G{\"u}nter Harder, and Don Zagier.
\newblock {\em The 1-2-3 of modular forms}.
\newblock Universitext. Springer-Verlag, Berlin, 2008.
\newblock Lectures from the Summer School on Modular Forms and their
  Applications held in Nordfjordeid, June 2004, Edited by Kristian Ranestad.

\bibitem{FreitagNote07}
E.~Freitag, 2007.
\newblock unpublished note on unitary groups.

\bibitem{FS}
Eberhard Freitag and Riccardo Salvati~Manni.
\newblock On siegel three folds with a projective calabi--yau model.
\newblock {\em Commun. Number Theory Phys.}, 5(3):713--750, 2011.

\bibitem{Hof11}
Eric Hofmann.
\newblock {\em Automorphic Products on Unitary Groups}.
\newblock PhD thesis, TU Darmstadt, 2011.
\newblock \url{http://tuprints.ulb.tu-darmstadt.de/2540/}.

\bibitem{Hof14}
Eric Hofmann.
\newblock Borcherds products on unitary groups.
\newblock {\em Mathematische Annalen}, 354:799--832, 2014.

\bibitem{Shimura}
Goro {Shimura}.
\newblock {\em {Introduction to the arithmetic theory of automorphic functions.
  Repr. of the 1971 orig.}}
\newblock Princeton, NJ: Princeton Univ. Press, repr. of the 1971 orig.
  edition, 1994.

\bibitem{Shin75}
Takuro Shintani.
\newblock On construction of holomorphic cusp forms of half integral weight.
\newblock {\em Nagoya Math. J.}, 58:83--126, 1975.

\bibitem{Wald78}
J.-L. Waldspurger.
\newblock Engendrement par des s\'eries th\^eta de certains espaces de formes
  modulaires.
\newblock {\em Invent. Math.}, 50(2):135--168, 1978/79.

\end{thebibliography}
\bigskip
\noindent \textit{%
Mathematisches Institut \\
Universität Heidelberg \\
Im Neuenheimer Feld 205 \newline {D-69120} Heidelberg\\
Germany} \newline
\href{mailto:hofmann@mathi.uni-heidelberg.de}{\nolinkurl{hofmann@mathi.uni-heidelberg.de}}.
\end{document}